 \documentclass[final,3p,times]{elsarticle}

\usepackage{amssymb}
\usepackage{amsthm}
\usepackage{amsmath}

\usepackage{graphicx}
\usepackage{subcaption}
\usepackage{amssymb}
\usepackage{xcolor}

\usepackage{scalerel}
\usepackage{tikz}
\usetikzlibrary{svg.path}

\definecolor{orcidlogocol}{HTML}{A6CE39}
\tikzset{
  orcidlogo/.pic={
    \fill[orcidlogocol] svg{M256,128c0,70.7-57.3,128-128,128C57.3,256,0,198.7,0,128C0,57.3,57.3,0,128,0C198.7,0,256,57.3,256,128z};
    \fill[white] svg{M86.3,186.2H70.9V79.1h15.4v48.4V186.2z}
                 svg{M108.9,79.1h41.6c39.6,0,57,28.3,57,53.6c0,27.5-21.5,53.6-56.8,53.6h-41.8V79.1z M124.3,172.4h24.5c34.9,0,42.9-26.5,42.9-39.7c0-21.5-13.7-39.7-43.7-39.7h-23.7V172.4z}
                 svg{M88.7,56.8c0,5.5-4.5,10.1-10.1,10.1c-5.6,0-10.1-4.6-10.1-10.1c0-5.6,4.5-10.1,10.1-10.1C84.2,46.7,88.7,51.3,88.7,56.8z};
  }
}

\newcommand\orcidicon[1]{\href{https://orcid.org/#1}{\mbox{\scalerel*{
\begin{tikzpicture}[yscale=-1,transform shape]
\pic{orcidlogo};
\end{tikzpicture}
}{|}}}}

\journal{}

\usepackage{hyperref}
\hypersetup{
	colorlinks=true,
	linkcolor={red!50!black},
	citecolor={blue!50!black},
	urlcolor={blue!20!black}
}

\usepackage[noabbrev,capitalize]{cleveref}

\begin{document}

\def\im{\mathrm{Im}\, }
\def\re{\mathrm{Re}\, }
\def\diam{\mathrm{diam}\, }
\def\mod{\mathrm{mod}\, }
\def\imod{\mathrm{invmod}\, }
\def\Hess{\mathrm{Hess}\, }
\def\rank{\mathrm{rank }\,  }
\def\rg{\mathrm{rg }\,  }
\def\d{\mathrm{d}}
\def\p{\partial }
\def\pr{\mathrm{pr}}
\def\reff{{\mathrm{ref}}}
\def\D{\mathrm{D}}
\def\id{\mathrm{id}}
\def\Id{\mathrm{Id}}
\def\Jet{\mathrm{Jet}}
\def\e{\epsilon}
\def\C{\mathbb{C}}
\def\Z{\mathbb{Z}}
\def\Q{\mathbb{Q}}
\def\N{\mathbb{N}}
\def\R{\mathbb{R}}
\def\K{\mathbb{K}}
\def\T{\mathbb{T}}
\def\diag{\mathrm{diag}\, }

	\newtheorem{theorem}{Theorem}[section]
	\newtheorem{example}{Example}[section]
	\newtheorem{definition}{Definition}[section]
	\newtheorem{remark}{Remark}[section]
	\newtheorem{corollary}[theorem]{Corollary}
	\newtheorem{lemma}[theorem]{Lemma}
	\newtheorem{proposition}[theorem]{Proposition}
	\newtheorem{prop}[theorem]{Proposition}
	\newtheorem{conjecture}[theorem]{Conjecture}
	\newtheorem{assumption}{Assumption}[section]

\begin{frontmatter}



\title{Variational Learning of Euler--Lagrange Dynamics from Data}

\author{Sina Ober-Blöbaum}
\author{Christian Offen\corref{cor1}}
\ead{christian.offen@uni-paderborn.de}
\ead[url]{https://www.uni-paderborn.de/en/person/85279}
\ead[orcid]{ORCiD: 0000-0002-5940-8057}
\cortext[cor1]{corresponding author ${\protect \orcidicon{0000-0002-5940-8057}}$}

\affiliation{organization={Paderborn University},
             addressline={Warburger Str.100},
             city={Paderborn},
             postcode={33098},
             country={Germany}}

\begin{abstract}
The principle of least action is one of the most fundamental physical principle.
It says that among all possible motions connecting two points in a phase space, the system will exhibit those motions which extremise an action functional.
Many qualitative features of dynamical systems, such as the presence of conservation laws and energy balance equations, are related to the existence of an action functional. Incorporating variational structure into learning algorithms for dynamical systems is, therefore, crucial in order to make sure that the learned model shares important features with the exact physical system.
In this paper we show how to incorporate variational principles into trajectory predictions of learned dynamical systems.
The novelty of this work is that (1) our technique relies only on discrete position data of observed trajectories.
Velocities or conjugate momenta do {\em not} need to be observed or approximated and {\em no} prior knowledge about the form of the variational principle is assumed. Instead, they are recovered using backward error analysis.
(2) Moreover, our technique compensates discretisation errors when trajectories are computed from the learned system. This is important when moderate to large step-sizes are used and high accuracy is required. For this, we introduce and rigorously analyse the concept of inverse modified Lagrangians by developing an inverse version of variational backward error analysis. (3) Finally, we introduce a method to perform system identification from position observations only, based on variational backward error analysis.
\end{abstract}


%
%
%
%
%
%
%
%
%


\end{frontmatter}



\section{Introduction}

Physics informed learning has sparked extensive interest since embedding prior knowledge about physical laws can greatly improve the quality of predictions of dynamics based on observation data.
One of the most fundamental principles is Hamilton's principle or the principle of least action. Among all curves connecting two points in a configuration space only those correspond to physical motions which extremise an action functional. The existence of an action functional induces some subtle structure on the differential equations governing the motions, which we will refer to as {\em variational structure}.
When trying to identify a system from data, it, therefore, appears to be advantageous to learn the action rather than the differential equations which govern the motions. This does not only reduce data requirements since only a scalar valued map rather than the right hand side of a differential equation needs to be learned, but, more importantly, it guarantees that the learned system shares qualitative aspects with physical systems. For instance, autonomous variational systems conserve energy. When an autonomous physical system is learned and the learned system is variational, then the learned system is conservative as well. 


Lagrangian Shadow Integration (LSI), introduced in this paper, provides a novel framework to predict motions of dynamical systems from data while incorporating variational structure. The data consists of discrete observations of trajectories (snapshot data). We assume that only position data is available, in contrast to other approaches such as Lagrangian Neural Networks (LNN) \cite{LNN}, which assume the availability of velocity and acceleration information. 
We incorporate variational structure as prior knowledge without making any assumptions on the specific form of the variational principle, in contrast to \cite{Aoshima2021}, for instance, which assumes a mechanical form.
The novelty of our approach lies in the fact that we do not learn the exact variational structure but an inverse modified structure. The inverse modified variational principle has the property that when integrated using a variational integrator {\em no} discretisation error occurs. We prove the existence of inverse modified variational structure (up to truncation error) and explain the compensation of discretisation errors. Our analysis heavily makes use of {\em variational backward error analysis}, which was developed in \cite{Vermeeren2017}.
Moreover, we show how to use Lagrangian Shadow Integration to identify structure and conserved quantities of dynamical systems based on observations of position data only, without any requirements of further knowledge about the system's structure.

In the proceeding of this section we will briefly review classical theory on Lagrangian dynamics and variational integrators, introduce the idea behind Lagrangian Shadow Integration, and contrast the technique to other approaches. We will restrict ourselves to autonomous systems in this work.
\Cref{sec:LSI} contains a more detailed explanation of the Lagrangian Shadow Integration framework. While Lagrangian Shadow Integration can be combined with many machine learning techniques, we show how to realise the technique using kernel based methods. \Cref{sec:ex} contains numerical examples. Finally, \cref{sec:Existence} provides rigorous theoretical underpinning of our technique and connects the framework with highly developed analysis techniques from Geometric Integration.

\subsection{Lagrangian dynamics and variational integrators}

In many dynamical systems
motions $q \colon [t_0,t_N] \to Q$ connecting two points $q_0$, $q_N$ of a state space $Q=\R^n$ are necessarily stationary points of an action functional
\begin{equation}\label{eq:ActionIntro}
S(q) = \int_{t_0}^{t_N} L(q(t),\dot{q}(t)) \d t,
\end{equation}
i.e.\ for any $v \colon [t_0,t_N] \to \R^n$ with $v(t_0)=0=v(t_N)$ we have
\[
\lim_{\e \to 0} \frac{S(q+ \e v) - S(q)}{\e} =0.
\]
Equivalently, motions $q$ fulfil the Euler--Lagrange equations
\begin{equation}\label{eq:ELIntro}
0 = \mathcal{E}(L) = \frac{\p L}{\p q}(q,\dot{q}) - \frac{\d }{\d t} \frac{\p L}{\p \dot{q}}(q,\dot{q})
\end{equation}
subject to the boundary conditions $q(t_0)=q_0$, $q(t_N)=q_N$.
$L$ is called the Lagrangian of the dynamical system.
In line with the examples presented in this article, we have restricted ourselves to the autonomous case.
The variational structure governing the motions is related to structural properties of the dynamical system. For instance, motions preserve the Hamiltonian
\begin{equation}
	\label{eq:HamIntro}
H(q, \dot{q}) = \sum_{j=0}^n \dot{q}^j \frac{\p L}{\p \dot{q}^j} - L(q,\dot{q}),
\end{equation}
which in classical mechanics corresponds to energy preservation in many cases.
Moreover, Noether's theorem relates variational symmetries of $L$ to further conserved quantities: for instance, rotational invariance of $L$ yields conservation of angular momenta. 

To guarantee that numerical solutions of \eqref{eq:ELIntro} inherit physical properties of the exact system related to variational structure, rather than applying a numerical integrator to \eqref{eq:ELIntro} directly, one discretises the action functional and looks for its stationary points. The discrete dynamics are governed by a discrete variational principle and a discrete version of Noether's theorem exists to derive discrete conservation laws \cite{GeomIntegration}. Integrators derived in such a way are called {\em variational integrators}. More precisely, if $t_0,t_1,\ldots,t_N$ is a uniform grid with step-size $h$ of the interval $[t_0,t_N]$ then $S$ in \eqref{eq:ActionIntro} is replaced by a discrete action of the form
\begin{equation}\label{eq:SDiscreteIntro}
S_\Delta ((q_j)_{j=0}^N) = \sum_{j=0}^{N-1}  L_\Delta (q_j,q_{j+1}),
\end{equation}
where the discrete Lagrangian $L_\Delta (a,b)$ is an approximation of the {\em exact discrete Lagrangian} $\int_{0}^{h} L(q(t),\dot{q}(t))$, where $q$ solves\footnote{A unique solution exists if $h>0$ is sufficiently small and $a$ and $b$ sufficiently close to one another.} \eqref{eq:ELIntro} subject to the boundary conditions $q(0)=a$, $q(h)=b$. Critical points $\nabla S_\Delta ((q_j)_{j=0}^N) =0$ of the discrete action \eqref{eq:SDiscreteIntro} fulfil the discrete Euler--Lagrange equations
\begin{equation}\label{eq:DELIntro}
	\nabla_2 L_\Delta (q_{j-1},q_j) + \nabla_1 L_\Delta (q_j,q_{j+1}) = 0, \qquad j=1,\ldots,N-1
\end{equation}
on inner grid points. Here $\nabla_1 L_\Delta$ and $\nabla_2 L_\Delta$ denote the derivative with respect to the first or second input argument of $L_\Delta$, respectively.

\begin{example}\label{ex:MPandTZ}
Popular discretisation methods are the midpoint rule
\[
L_\Delta(q_0,q_1) = h L\left(\frac{q_0+q_1}{2},\frac{q_1-q_0}{h}\right)
\]
or the trapezoidal rule
\[
L_\Delta(q_0,q_1) =  \frac{h}{2} L\left(q_0,\frac{q_1-q_0}{h} \right) + \frac{h}{2} L\left(q_1,\frac{q_1-q_0}{h} \right).
\]
Both yield second order schemes.
\end{example}

To compute an approximation to a solution to an initial value problem $q(t_0)=q_0$, $\dot{q}(t_0)=\dot{q}_0$ for \eqref{eq:ELIntro}, variational integration proceeds as follows.

\begin{itemize}
\item
A conjugate momentum $p_0 = \frac{\p L}{\p \dot{q}}(q_0,\dot{q}_0)$ is computed (Legendre transform).

\item
$q_1$ is approximated by solving the discrete conjugate momenta equation $-p_0 = \nabla_1 L_\Delta (q_0,q_1)$ for $q_1$.

\item
The discrete Euler--Lagrange equations \eqref{eq:DELIntro} is used to compute $q_2$, $q_3$, $\ldots$ successively.

\item
If of interest, discrete conjugate momenta can be computed using
\[
p_j = -\nabla_1 L_\Delta (q_j,q_{j+1}) 
\quad \text{ or } \quad p_{j+1} = \nabla_2 L_\Delta (q_j,q_{j+1})
\]
and velocities can be computed by solving $p_j = \frac{\p L}{\p \dot{q}}(q_j,\dot{q}_j)$ for $\dot{q}_j$ in a post-processing step.

\end{itemize}

\subsection{Lagrangian Shadow Integration}

Lagrangian Shadow Integration provides a framework to predict motions of unknown dynamical systems with variational structure form trajectory observations. It consists of a step in which the underlying continuous dynamical system is identified using machine learning techniques and another step in which the system is integrated numerically. Both steps make use of the system's variational structure.
We assume that discrete snapshot data of trajectories to a time step $h$ are available. Lagrangian Shadow Integration proceeds as follows.

\begin{enumerate}

\item
Choose a variational integrator.

\item
Learn an inverse modified Lagrangian $L_\imod$ of the system. This is not the exact Lagrangian $L$ of the dynamical system but an adapted (shadow) version of $L$ which is tailored to the variational integration scheme.

\item
Apply the variational integration technique to $L_\imod$ to compute trajectories and, if needed, conjugate momenta.

\item
Bonus. Identify the exact Lagrangian $L$ and energy $H$ using backward error analysis techniques from the learned $L_\imod$. Compute velocities to the trajectory data of step (3).

\end{enumerate}

The inverse modified Lagrangian $L_\imod$ can be learned directly from the observed snapshot data. Moreover, it has the property that in step (3) the discretisation error of the variational integrator is compensated, which yields excellent conservation properties.
The concept of inverse modified Lagrangians $L_\imod$, introduced in this article, can be thought of as an inverse version of the concept of modified Lagrangians $L_\mod$ \cite{Vermeeren2017}, which will be reviewed in the next section. In short, a {\em modified Lagrangian} governs the numerical solutions obtained by a variational integration scheme applied to a system defined by a Lagrangian $L$ (the exact system). An {\em inverse modified Lagrangian} is a continuous Lagrangian $L_\imod$ such that a variational integration scheme applied to $L_\imod$ will produce motions of the system defined by $L$ (the exact system).

Lagrangian Shadow Integration can be contrasted to other approaches of predicting the motions of dynamical systems from data.


\begin{itemize}

\item
An incorporation of variational structure in prediction model guarantees that structural properties of the dynamics are preserved such that predicted motions show an excellent energy behaviour, for instance. Moreover, Lagrangian Shadow Integration could be combined with learning techniques which guarantee known symmetry properties of $L_\imod$ such as GIM kernels \cite{ridderbusch2021learning} or group invariant neural networks \cite{Cohen2019,Cohen2016,Kondor2018a,Dehmamy2021}.
In this way, if symmetries of the system are known, the discrete Noether theorem guarantees the existence of corresponding conservation laws.
Moreover, we only learn the scalar valued map $L_\imod$ rather than the flow map $(x,v) \mapsto \Phi(x,v)$ of the underlying second order differential equation. This reduces the dimensionality of the problem.

\item
The observed data only contains position data without velocity information. Any technique that aims to learn the flow map $(x,v) \mapsto \Phi(x,v)$ of the underlying second order differential equation but does not make use of prior knowledge about variational structure would have to approximate velocities or conjugate momenta before teaching a model for $\Phi$.
The approximation of unknown initial velocities of the observed trajectories introduces an additional approximation error in the data. In contrast, Lagrangian Shadow Integration can directly use the observed data without additional approximations or without requiring further prior knowledge. This contrasts Lagrangian Shadow Integration to other techniques which
\begin{itemize}
	\item
assume that velocity data or even higher derivatives of trajectories is available \cite{LNN}

\item
or assume more prior knowledge about the system such as \cite{Aoshima2021,Evangelisti2022}. Prior knowledge about the form of the Lagrangian or kinetic term fix much of the phase space structure. Lagrangian Shadow Integration can discover the full Lagrangian from data without restrictions on its form.

\end{itemize}

\item
Moreover, approaches such as Lagrangian Neural Networks \cite{LNN} and related approaches \cite{Cheng2016,Aoshima2021,Evangelisti2022} aim at learning the exact Lagrangian. When such a system is integrated, an additional approximation error occurs which Shadow Lagrangian Integration compensates by construction.
Moreover, since a Lagrangian is not uniquely determined by the motions of a dynamical system, techniques which aim to identify the exact Lagrangian might learn a Lagrangian which is not well suited for numerical integration, as we will demonstrate in \cref{sec:ex}.

\item
Shadow Lagrangian Integration allows the employment of established analysis tools from geometric integration technique. These can be actively employed for verification and system identification.

\item


Another approach is to learn the Hamiltonian of a dynamical system instead of its Lagrangian. In contrast to Lagrangian Shadow Integration, these approaches typically assume that the symplectic structure of the system is known and assume that observations of momentum data is available in addition to position data: Hamiltonian Neural Networks \cite{HNN} learn the Hamiltonian of a system from position and momentum data of trajectories. Bertalan et al.\ \cite{Bertalan2020} learn Hamiltonians using Gaussian Processes and introduce a technique to learn the symplectic structure of a system from position and momentum observations of trajectories. SympNets \cite{SympNets} learn the symplectic flow map of a Hamiltonian systems using a neural network architecture which is symplectic by construction. Other approaches learn the generating function of the symplectic flow map \cite{Rath2021}. In a previous article \cite{symplecticShadowIntegrators} the authors introduced {\em Symplectic Shadow Integration} which learns an inverse modified Hamiltonian using kernel based methods. The present article can be seen as a Lagrangian version of \cite{symplecticShadowIntegrators}.

\item
Backward error analysis has been used to improve the accuracy of classical numerical methods for solving differential equations in \cite{Chartier2007}, for instance.
For Hamiltonian systems it has been used in a data-driven context in \cite{symplecticShadowIntegrators,David2021}.
While traditional backward error analysis of structure preserving methods \cite{GeomIntegration} applies to the Hamiltonian description of dynamical systems, here we apply {\em variational backward error analysis}, which stays fully on the variational side and was developed in \cite{Vermeeren2017}. In this work, we develop an inverse version of variational backward error analysis and apply it in a data-driven context. (Inverse) variational backward error analysis is used as a tool to analyse and verify the learning methods, for system identification, and for the computation of velocities to trajectories for which only position data has been observed.

\item 
Up to a coordinate transform, learning an inverse modified Lagrangian $L_\imod$ from data, as proposed in Lagrangian Shadow Integration, corresponds to learning a discrete Lagrangian $L_\Delta$, as done in \cite{Qin2020}.
This is sufficient for solving boundary value problems or to extend discrete trajectories. However, to solve initial value problems with initial data of the form $(q_0,\dot{q}_0)$ or to perform system identification, variational backward error analysis as performed in Lagrangian Shadow Integration, is required.

\end{itemize}


\section{The Lagrangian Shadow Integration Framework}\label{sec:LSI}

After reviewing the concept of modified Lagrangians, we introduce inverse modified Lagrangians and show how to predict motions once an inverse modified Lagrangian is learned. We then discuss the non-uniqueness of Lagrangians and develop regularisation terms that are needed in the learning process. Finally, we show how to learn inverse modified Lagrangians using neural networks as well as Gaussian processes. 

\subsection{Modified Lagrangians}

Let us review the concept of modified Lagrangians that were introduced in \cite{Vermeeren2017}. Related computational tools are reviewed in more detail in \cref{sec:Existence}.

When a dynamical system defined by a Lagrangian $L(q,\dot{q})$ is integrated using a variational integrator with step-size $h$, then numerical solutions fulfil discrete Euler--Lagrange equations
\[
	\nabla_2 L_\Delta (q_{j-1},q_j) + \nabla_1 L_\Delta (q_j,q_{j+1}) = 0, \qquad j=1,\ldots,N-1,
\]
where $L_\Delta$ is the discrete Lagrangian corresponding to the integrator.
Variational backward error analysis seeks a continuous Lagrangian $L_\mod(q,\dot{q})$ such that any solution $q$ to the modified Euler--Lagrange equations $\mathcal{E}(L_\mod)=0$ fulfils the discrete Euler--Lagrange equations, i.e.\
\begin{equation}\label{eq:DELContinuousq}
\nabla_2 L_\Delta (q(t-h),q(t)) + \nabla_1 L_\Delta (q(t),q(t+h)) = 0.
\end{equation}
Modified Lagrangians $L_\mod$ are sought as formal power series in the step size $h$ such that when truncated to any order in $h$ then \eqref{eq:DELContinuousq} is fulfilled up to higher order terms. Except in special cases, the formal power series $L_\mod$ will not converge. However, optimal truncation results have been proved for an analogous theory on the Hamiltonian side \cite{GeomIntegration} and a dynamical meaning of truncations of $L_\mod$ can be observed in numerical experiments \cite{BEAMulti,BEASymPDE}.
The relation between $L$, $L_\Delta$, and $L_\mod$ is illustrated in \cref{fig:LmodIllustration}.
\begin{figure}
	\hfill
	\includegraphics[width=0.4\paperwidth]{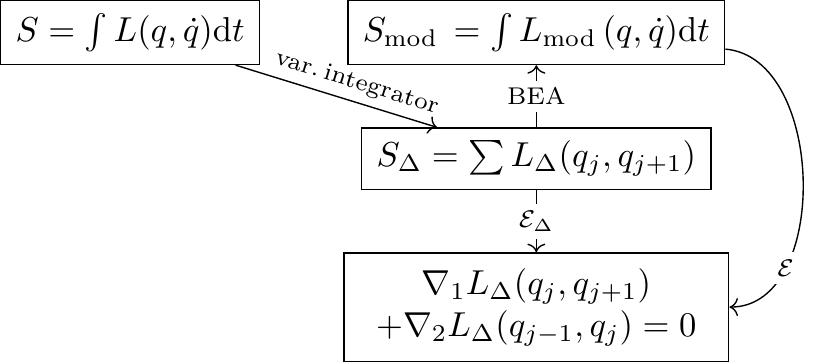}
	\caption{
		Variational backward error analysis (BEA) seeks a modified Lagrangian $L_\mod$ such that solutions to the modified Euler--Lagrange equations $\mathcal{E}(L_\mod)=0$ fulfil the discrete Euler--Lagrange equations $\mathcal{E}_\Delta(L_\Delta)=0$ with $q_j = q(t+jh)$.
	}\label{fig:LmodIllustration}
\end{figure}

\subsection{The concept of inverse modified Lagrangians}

In the following, we consider (consistent) variational discretisation schemes $L \mapsto L_\Delta$, which map a Lagrangian $L$ to a discrete Lagrangian $L_\Delta$.
We further require that the discretisation scheme is linear in the Lagrangian, i.e.\ $sL^1 + L^2 \mapsto s L^1_\Delta+L^2_\Delta$.

The {\em exact discrete Lagrangian} $L_\Delta^{\mathrm{ex}}$ to a step size $h$ and a system defined by a Lagrangian $L$ is defined as
\begin{equation}\label{eq:ExactDiscreteL}
L_\Delta^{\mathrm{ex}}(q_0,q_1)= \int_{t_0}^{t_0+h} L(q(t),\dot{q}(t)) \d t,
\end{equation}
where $q$ solves the Euler--Lagrange equations \eqref{eq:ELIntro} subject to $q(t_0)=q_0$, $q(t_0+h)=q_1$.
The inverse modified Lagrangian to a Lagrangian $L$ and an integration scheme is a continuous Lagrangian $L_\imod$ such that $L_{\imod \Delta}$ coincides with the exact discrete Lagrangian $L_\Delta^{\mathrm{ex}}$ \eqref{eq:ExactDiscreteL}, i.e.\
\begin{equation}\label{eq:LinvmodLexact}
	L_{\imod\Delta}(q_0,q_1) = L_\Delta^{\mathrm{ex}}(q_0,q_1).
\end{equation}
As the discussion in \cref{sec:Existence} will show, $L_\imod$ exists only as a formal power series in the step size $h$. Equality in \eqref{eq:LinvmodLexact} means that the left and right hand side coincide up to any order in $h$.
Again, in the sense of formal power series, the discrete conjugate momenta of $L_{\imod\Delta}$ coincide with the exact conjugate momenta of $L$, i.e.\
\begin{align}
	\nabla_1 L_{\imod\Delta} (q_0,q_1)
	&= \frac{\p L_{\imod\Delta} }{\p q_0}(q_0,q_1)
	= -\frac{\p L}{\p \dot{q}}(q_0,\dot{q}_0)
	= -p(t_0)\\
	\nabla_2 L_{\imod\Delta} (q_0,q_1)
	&= \frac{\p L_{\imod\Delta} }{\p q_1}(q_0,q_1)
	= \phantom{-}\frac{\p L}{\p \dot{q}}(q_1,\dot{q}_1) 
	= \phantom{-}p(t_1).
\end{align}
The corresponding calculation for exact discrete Lagrangian can be found in \cite{GeomIntegration,MarsdenWestVariationalIntegrators}, for instance.
The exact Lagrangian $L$ can be recovered from $L_\imod$ to any order in $h$ using variational backward error analysis \cite{Vermeeren2017}. The necessary computational tools will be reviewed in \cref{thm:VarBEA}. For instance, for the midpoint rule in the one-dimensional case we have
\begin{equation}\label{eq:MPBEA}
	L = L_{\imod}+\frac{1}{24} h^2 \left(\frac{\left(L_{\imod}^{(1,0)}-L_\imod^{(1,1)} \dot{q}\right)^2}{L_{\imod}^{(0,2)}}-L_{\imod}^{(2,0)} \dot{q}^2\right)
	+O\left(h^4\right).
\end{equation}
For the trapezoidal rule we obtain
\[
L = L_{\imod} + \frac{1}{24} h^2 \left(2 L_{\imod}^{(2,0)} \left(\dot{q}\right)^2+\frac{\left(L_{\imod}^{(1,0)}-L_{\imod}^{(1,1)} \dot{q}\right)^2}{L_{\imod}^{(0,2)}}\right)
+O\left(h^4\right).
\]
Here $L^{(i,j)} = \frac{\p^{i+j}}{ (\p q)^i  (\p \dot{q})^j } (q,\dot{q})$ denote derivatives which can be computed using automatic differentiation.
Higher order terms and higher dimensional examples can be found or generated using the accompanying Mathematica scripts \cite{LagrangianShadowIntegratorSoftware}.

The relation between a Lagrangian, a modified Lagrangian, and an inverse modified Lagrangian is illustrated in \cref{fig:LimodDiagram}.
\begin{figure}
\centering
\includegraphics[width=0.8\paperwidth]{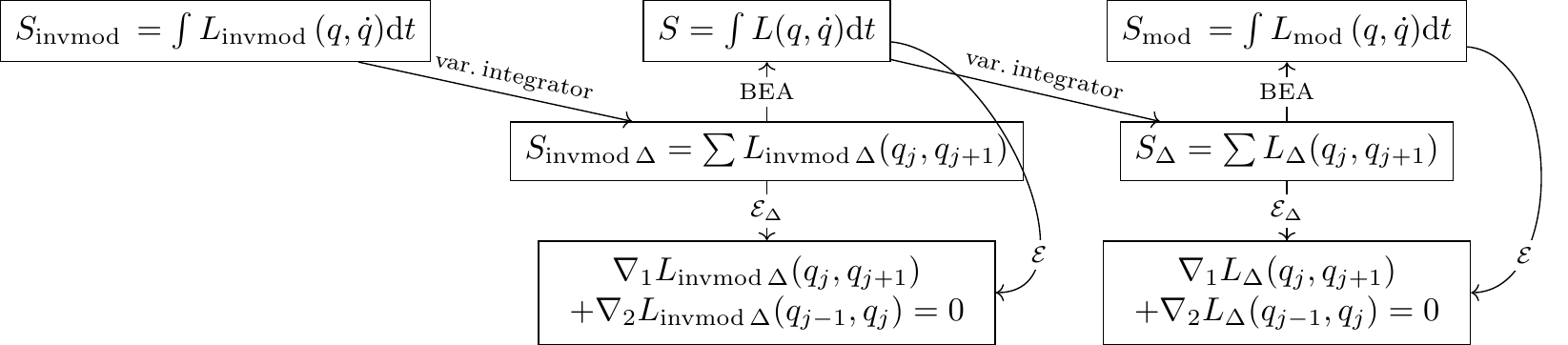}
\caption{
This figure extends \cref{fig:LmodIllustration} by a column to the left and two boxes in the centre column.
To a given variational integration scheme, the inverse modified Lagrangian $L_\imod$ is a continuous Lagrangian such that the discrete Lagrangian $L_{\imod \Delta}$ to $L_\imod$ coincides with the exact Lagrangian $L$ up to any order in the step size $h$. The exact Lagrangian $L$ can be recovered from $L_\imod$ using variational backward error analysis (BEA). Up to any order in the step size $h$, solutions to the exact Euler--Lagrange equations $\mathcal{E}(L)=0$ fulfil the discrete Euler--Lagrange equations $\mathcal{E}_\Delta(L_\imod)=0$ to the Lagrangian $L_\imod$ with $q_j = q(t+jh)$.
}\label{fig:LimodDiagram}
\end{figure}



\subsection{Computation of motions and idea of Lagrangian Shadow Integration}

The idea of Lagrangian Shadow Integration is to learn $L_\imod$ such that the Euler--Lagrange equations to $L_{\imod \Delta}$ are consistent with the observed trajectory data. For this, we will ignore the fact that $L_\imod$ only exists in the sense of formal power series and simply model $L_\imod$ as a neural network or Gaussian Process which maps $(q,\dot{q})$ to a scalar output.
Once $L_\imod$ is learned, an approximation $L$ to the (unknown) exact Lagrangian is computed using a variational backward error analysis formula (such as \eqref{eq:MPBEA} in case of the midpoint rule in one dimension).
A trajectory to initial position and velocity $(q_0,\dot{q}_0)$ can then be computed as follows.

\begin{itemize}

\item
Set $p_0 = \frac{\p L}{\p \dot{q}}(q_0,\dot{q}_0)$.

\item
Solve $\nabla_1 L_{\imod\Delta} (q_0, q_1) + p_0 = 0$ for $q_1$. 

\item 
For $j=1,\ldots,N-1$ solve the discrete Euler--Lagrange equations for $L_{\imod \Delta}$
\begin{equation}\label{eq:DELDelta}
	\nabla_2 L_{\imod\Delta} (q_{j-1},q_j) + \nabla_1 L_{\imod\Delta} (q_j,q_{j+1}) = 0, \qquad j=1,\ldots
\end{equation}
for $q_{j+1}$.

\item
If a velocity value $\dot{q}_j$ is needed, it can be computed by solving
\[\nabla_2 L_{\imod\Delta} (q_{j-1}, q_j)  = \frac{\p L}{\p \dot{q}}(q_j,\dot{q}_j)\]
for $q_j$, where a formula for $L$ has been calculated from $L_\imod$ to the desired order in $h$ using variational backward error analysis.

\end{itemize}

\begin{remark}
Notice that values for the conjugate momenta
\[
\frac{\p L}{\p \dot{q}}(q_j,\dot{q}_j) = p_j 
= - \nabla_1 L_\imod (q_j,q_{j+1}) 
= \nabla_2 L_\imod (q_{j-1},q_{j})
\]
do not necessarily carry a physical meaning since they depend on the choice of a Lagrangian, which is not uniquely determined by a system's dynamics. 
\end{remark}

\subsection{Non-uniqueness of Lagrangians}\label{sec:ambiguity}
Lagrangians are not uniquely determined by the system's dynamics. In particular, constant Lagrangians have trivial Euler--Lagrange equations $\mathcal{E}(L)\equiv 0$ and are, therefore, consistent with any dynamical data. To avoid learning singular Lagrangians, we develop a normalisation condition for $L$ and $L_\imod$. For this, let us review the non-uniqueness of the Lagrangian $L$ with its conjugate momenta $p$ and also discuss the ambiguity of symplectic structures and Hamiltonians.

Lagrangian dynamics can be translated to Hamiltonian dynamics. For this, conjugate momenta are introduced as $p = \frac{\p L}{\p \dot{q}}(q,\dot{q})$. Provided that $L$ is regular, i.e.\ $\left(\frac{\p^2 L}{\p \dot{q}^j \p \dot{q}^i}\right)_{i,j = 1}^n$, $(q,\dot{q})$ and $(q,p)$ provide local coordinate systems on the tangent bundle $TQ$ and the cotangent bundle $T^\ast Q$, respectively, which are both locally identified with $\R^{2n}$ in the following.
The Hamiltonian is given as
\[H(q,p) = \sum_{j=1}^n \dot{q}^j((q,p)) p^j - L(q,\dot{q}(q,p)),\]
where $q^j$ is expressed in terms of $q,p$. The Hamiltonian $H$ is conserved along motions.
The corresponding symplectic structure $\omega$ and Liouville volume $\mu$ are given as
\[
\omega = \sum_{j=1}^n \d q^j \wedge \d p^j, \quad \mu = \frac{1}{n!} \underbrace{\omega \wedge \ldots \wedge \omega}_{n\text{ times}}.
\]
In coordinates $(q,\dot{q})$ the Liouville volume form can be expressed as
\[
\mu = \frac{\p L}{\p \dot{q}^1} \cdot \ldots \cdot  \frac{\p L}{\p \dot{q}^n} \d q^1 \wedge \d \dot{q}^1 \wedge \ldots \wedge \d q^n \wedge \d \dot{q}^n.
\]
The Lagrangian $L$ is defined by the dynamics up to scalar multiples and up to a total derivative of a scalar valued map $\xi$ in $q$, i.e.\ the Lagrangian $\tilde{L} = s L + \frac{\d }{\d t} \xi(q)$ defines the same dynamics as $L$ but its conjugate momenta is $\tilde{p} = s p + \nabla \xi$. The symplectic structures $\omega = \d q \wedge \d p$ and $\tilde{\omega} =  \d q \wedge \d \tilde{p}$ relate as $\tilde{\omega} = s \omega$. Once the symplectic structure is fixed, the Hamiltonian $H$ is defined up to a constant by the dynamics.

In the following, we will consider the following non-triviality condition for $L_\imod$, where $E^{2n}=([0,1] \times [0,1])^n$ is the unit hypercube in $\R^{2n}$ with coordinates $q_1,\dot{q}_1,\ldots,q_n,\dot{q}_n$:

\begin{equation}\label{eq:normaliseL}
c=\int_{E^{2n}} \mu = \int_{E^{2n}} \frac{\p L}{\p \dot{q}^1} \cdot \ldots \cdot  \frac{\p L}{\p \dot{q}^n} \d q^1  \d \dot{q}^1 \ldots  \d q^n  \d \dot{q}^n,
\end{equation}
where $c \not =0$ is a constant controlling the scaling of $L$. Notice that while \eqref{eq:normaliseL} removes the scaling ambiguity, it does not remove the ambiguity of summands which are total derivatives, even if a further normalisation such as $L(0,0)=0$ is added. However, it forces $L$ to be non-trivial and regular.
In the following, we use the approximation
\begin{equation}\label{eq:normaliseLimod}
	c=\int_{E^{2n}} \mu = \int_{E^{2n}} \frac{\p L_\imod}{\p \dot{q}^1} \cdot \ldots \cdot  \frac{\p L_\imod}{\p \dot{q}^n} \d q^1  \d \dot{q}^1 \ldots  \d q^n  \d \dot{q}^n
\end{equation}
of \eqref{eq:normaliseL}, which coincides with \eqref{eq:normaliseL} up to $\mathcal{O}(h)$-terms (in the infinite data limit).

\begin{remark}
An alternative (more ad hoc) non-triviality condition is
\[  \frac{\p L_\imod}{\p \dot{q}}(q_{\mathrm{base}},\dot{q}_{\mathrm{base}})  - p_{\mathrm{base}} =0, \qquad p_{\mathrm{base}} \not = 0,\]
which imposes that the value of the conjugate momentum of $L_\imod$ at an arbitrary point $(q_{\mathrm{base}},\dot{q}_{\mathrm{base}})$ is $p_{\mathrm{base}}$. Again, this does not impose any restrictions on the dynamical system. 
However, in the following we will use \eqref{eq:normaliseLimod} as it balances several points of the phase space and is more geometric.
\end{remark}

\subsection{Learning the inverse modified Lagrangian}\label{sec:learn}
To teach a model for $L_\imod$, the idea is to use the discrete Euler--Lagrange equations \eqref{eq:DELDelta} for all available data triples $(q_{j-1},q_{j},q_{j+1})$ together with a discretisation of the non-triviality condition \eqref{eq:normaliseLimod}. In addition, we will use a simple quadrature on \eqref{eq:normaliseLimod} and obtain the condition $\ell_{\mathrm{nontrivial}} =0$ with
\begin{equation}\label{eq:normaliseLimodQuad}
\ell_{\mathrm{nontrivial}} = c - \frac{1}{2^{2n}}\sum_{\mathrm{corners\, of\,} E^{2n}} \frac{\p L_\imod}{\p \dot{q}^1} \cdot \ldots \cdot  \frac{\p L_\imod}{\p \dot{q}^n}.
\end{equation}
Without the condition $\ell_{\mathrm{nontrivial}}=0$ any constant inverse modified Lagrangian $L_\imod \equiv \mathrm{const}$ would be consistent with the data because the discrete Euler--Lagrange equations are fulfilled on the data for any $L_{\imod\Delta} \equiv \mathrm{const}$.

\subsubsection{ANN}
If $L_\imod$ is modelled as a differentiable artificial neural network, we can define \[
\ell_{\mathrm{data}} 
=  \sum_{q \in \mathrm{trj}} \sum_{j=1}^{ \mathrm{length}(q)-1} \| 
	\nabla_2 L_{\imod \Delta} (q_{j-1},q_j) + \nabla_1 L_{\imod \Delta} (q_j,q_{j+1}) \|^2.
\]
Here $\mathrm{trj}$ denotes the set of all trajectories of the training data. Minimisation of $\ell_{\mathrm{data}}$ enforces the discrete Euler--Lagrange equations to hold on the training data. More concretely, for the midpoint rule $\ell_{\mathrm{data}}$ reads
\begin{equation}
\label{eq:ElDataNN}
\begin{split}
\ell_{\mathrm{data}} =  \sum_{q \in \mathrm{trj}} \sum_{j=1}^{ \mathrm{length}(q)-1} &\Bigg\| 
\frac 12 \frac{\p L_{\imod}}{\p q} \left( \frac{q_{j-1}+q_j}{2}, \frac{q_{j}-q_{j-1}}{h} \right)
+\frac 1h \frac{\p L_{\imod}}{\p \dot q} \left( \frac{q_{j-1}+q_j}{2}, \frac{q_{j}-q_{j-1}}{h} \right)\\
&+ \frac 12 \frac{\p L_{\imod}}{\p q} \left( \frac{q_{j}+q_{j+1}}{2}, \frac{q_{j+1}-q_j}{h} \right) 
- \frac 1h \frac{\p L_{\imod}}{\p \dot q} \left( \frac{q_{j}+q_{j+1}}{2}, \frac{q_{j+1}-q_j}{h} \right) 
\Bigg\|^2.
\end{split}
\end{equation}

The network can then be trained using a weighted sum of the non-triviality term $\|\ell_{\mathrm{nontrivial}}\|^2$ and the term $\ell_{\mathrm{data}}$. The term $\|\ell_{\mathrm{nontrivial}}\|^2$ can be interpreted as a regularising term.

\subsubsection{Gaussian Process Regression}\label{subsec:GPLearnimodL}
In the numerical examples of this paper, we will use Gaussian Process Regression. For an introduction we refer to \cite{Rasmussen2005}. The method can be applied to all kernel based methods with sufficiently differentiable kernel.

In the following, denote the tangent bundle over the state space $Q=\R^n$ by $TQ$ which is identified with $\R^{2n}$.
%
%
If $\hat{L}_\imod$ is a Gaussian process with zero mean and kernel $k\colon TQ \times TQ \to \R$ and (a priori unknown) values $L_\imod(Z) = \begin{pmatrix}
L_\imod(z_1),& \ldots, &L_\imod(z_M) \end{pmatrix}^\top$ for $ Z = \begin{pmatrix}
z_1,& \ldots, & z_M \end{pmatrix} \subset TQ$, then the value $L_\imod(y)$ for some $(q,\dot{q}) \in TQ$ can be predicted as the expectation
\begin{equation}\label{eq:LimodExpect}
\mathbb{E}[\hat{L}_\imod((q,\dot{q}))|(Z,L_\imod(Z))] = k((q,\dot{q}),Z)^\top k(Z,Z)^{-1} L_\imod(Z)
\end{equation}
with $k((q,\dot{q}),Z)^\top = \begin{pmatrix}
k((q,\dot{q}),z_1), & \ldots, & k((q,\dot{q}),z_M)
\end{pmatrix}$ and $k(Z,Z) = (k(z_i,z_j))_{i,j=1}^M$.


Since the discretisation scheme $I$ is linear in the Lagrangian, we have
\begin{equation}\label{eq:DLimoddiscExpect}
\begin{split}
L_{\imod\Delta}((q_a,q_b)) = \mathbb{E}[\hat{L}_{\imod\Delta}&((q_a,q_b))|(Z,L_\imod(Z))] 
=I[k]((q_a,q_b),Z)^\top k(Z,Z)^{-1} L_\imod(Z)
\end{split}
\end{equation}
In case of the midpoint rule
\[I[k]((q_a,q_b),Z) = \begin{pmatrix}
k \left(\left( \frac{q_a+q_b}{2},\frac{q_b-q_a}{h} \right),z_1 \right),& \ldots,& k\left(\left( \frac{q_a+q_b}{2},\frac{q_b-q_a}{h} \right),z_M \right)
\end{pmatrix}.\]
In case of the trapezoidal rule
\[I[k]((q_a,q_b),Z) =  \left(
	\frac{1}{2}k \left(\left( q_a,\frac{q_b-q_a}{h} \right),z_j \right) + \frac{1}{2}
	k \left(\left( q_b,\frac{q_b-q_a}{h} \right),z_j \right)
	 \right)_{j=1}^M.
\]
In general, $I[k]((q_a,q_b),Z) = \left(
I(y \mapsto k(y,z_j))(q_a,q_b)\right)_{j=1}^M$.
The right hand side of \eqref{eq:DLimoddiscExpect} can now be differentiated with respect to $q_a$ or $q_b$ to obtain expressions for
$\nabla_1 L_{\imod\Delta}((q_a,q_b))$
and
$\nabla_2 L_{\imod\Delta}((q_a,q_b))$.
These can be substituted into the discrete Euler--Lagrange equations \eqref{eq:DELDelta} for each data triple $(q_{j-1},q_j,q_{j+1})$. 
For the midpoint rule this yields a linear system consisting of the following equation for each data triple $(q_{j-1},q_j,q_{j+1})$ in the training set.
\begin{align*}
\Bigg[
&\frac{1}{h} \left( \nabla_{\dot{q}} k \left(\left( \frac{q_{j-1}+q_j}{2},\frac{q_j-q_{j-1}}{h} \right),Z  \right)^\top
- \nabla_{\dot{q}} k \left(\left( \frac{q_{j+1}+q_{j}}{2},\frac{q_{j+1}-q_j}{h} \right),Z  \right)^\top
\right)\\
+&\frac{1}{2} \left( \nabla_{q} k \left(\left( \frac{q_{j-1}+q_j}{2},\frac{q_j-q_{j-1}}{h} \right),Z  \right)^\top
- \nabla_{q} k \left(\left( \frac{q_{j+1}+q_{j}}{2},\frac{q_{j+1}-q_j}{h} \right),Z  \right)^\top
\right)
\Bigg]
\cdot k(Z,Z)^{-1}L_\imod(Z) =0
\end{align*}

Together with the non-triviality condition $\ell_{\mathrm{nontrivial}} =0$
\[
\left(\frac{1}{2^{2n}} \sum_{v \in \mathrm{corners}(E^{2n})} \mathbf{1}_n^\top \nabla_{\dot{q}} k(v,Z)^\top \right) k(Z,Z)^{-1}L_\imod(Z) =0
\]
where $\mathbf{1}_n = \begin{pmatrix}	1,& \ldots, &1\end{pmatrix}$
and the normalising condition
\[
k((q^\ast,\dot{q}^\ast),Z)^\top k(Z,Z)^{-1} L_\imod(Z) = 0
\]
for some $(q^\ast,\dot{q}^\ast) \in TQ$,
the system consists of $n \cdot K + 2$ equations for $M$ unknowns. Here $K$ is the number of triples $(q_{j-1},q_j,q_{j+1})$ from the training data. 
Now a least-square or minimal norm solution of the linear system is computed. This means we have computed predictions for $k(Z,Z)^{-1} L_\imod(Z)$ at points $ Z = \begin{pmatrix}	z_1,& \ldots, & z_M \end{pmatrix} \subset TQ$ based on the available training data triples on $Q$.

In computations, we chose $Z$ to consists of the points in $TQ$ for which the training data provides maximal information and then predict any further values of $L_\imod$ or its derivatives using \eqref{eq:LimodExpect} or derivatives thereof.
In case of the midpoint rule, we chose $Z= \left( \frac{q_j+q_{j+1}}{2} , \frac{q_{j+1}-q_j}{h} \right)_{(q_j,q_{j+1})}$ where $(q_j,q_{j+1})$ runs over all data pairs in the training data set.

\begin{remark}
Notice that it suffices to solve the linear system $(n \cdot K + 2) \times M$ dimensional system for the product $B= k(Z,Z)^{-1} L_\imod(Z)$ and use this expression in \eqref{eq:LimodExpect} or derivatives thereof. A computation of $k(Z,Z)^{-1}$ is not needed.
\end{remark}


\begin{remark}
	As there are more pairs $(q_j,q_{j+1})$ than triples $(q_{j-1},q_j,q_{j+1})$ in the training data set, the (non-homogeneous) linear system is underdetermined.\footnote{Leaving aside some obscure cases.}
	Moreover, the non-triviality condition and normalisation does not fully remove the ambiguity of Lagrangians explained in \cref{sec:ambiguity}. 
	Assuming the existence\footnote{In \cref{sec:Existence} we will prove that $L_\imod$ exists as a formal power series in $h$. Motivated by classical results in backward error analysis (BEA), we conjecture optimal truncation results such that $L_\imod$ exists up to exponentially small terms in the step-size $h$.} of $L_\imod$, we expect the system to be of rank smaller than $n \cdot K + 2$ and compute a minimal norm solution.
\end{remark}

\subsubsection{Comparison - Lagrangian Gaussian Process (LGP)}\label{sec:LGP}

For comparison, we also develop a method which can learn the Lagrangian when not only position but also velocity and acceleration data has been observed. The following can be seen as a Gaussian process version of Lagrangian neural networks (LNN) \cite{LNN}. However, in contrast to the approach taken in \cite{LNN}, our algorithm avoids matrix inversion. In \cref{fig:LimodDiagram} LGP corresponds to learning $L$ with motions computed using the discrete Euler--Lagrange equation $\mathcal{E}_\Delta(L_\Delta)=0$.

\begin{enumerate}
	
	\item
	We model a Lagrangian as a Gaussian Process and impose the (continuous) Euler--Lagrange equations at the training data
	\begin{equation}\label{eq:ELHess}
		\mathcal{E}(L)^k = \frac{\p L}{ \p \dot{q}^k} - \frac{\d }{\d t} \frac{\p L}{ \p \dot{q}^k} = \frac{\p L}{ \p \dot{q}^k} - \sum_{j=0}^n \left( \frac{\p^2 L}{ \p q^j \p \dot{q}^k} \dot{q}^j - \frac{\p^2 L}{ \p \dot{q}^j \p \dot{q}^k} \ddot{q}^j\right)
	\end{equation}
	together with the non-triviality condition and normalisation that we employed for learning the inverse modified Lagrangian. More precisely,
	if the training data is presented as $Z=(z_1,\ldots,z_K) = ((q_1,\dot{q}_1),\ldots,(q_K,\dot{q}_K))$ then
	\[
	r_{j,i}=\nabla_{\dot{q}} k((q_j,\dot{q}_j),z_i)-
	\D^2_{q,\dot{q}} k((q_j,\dot{q}_j),z_i) \dot{q}_j - \D^2_{\dot{q},\dot{q}} k((q_j,\dot{q}_j),z_i) \ddot{q}_j.
	\]
	is a column vector of length $n$. Here $\D^2_{q,\dot{q}} k((q_j,\dot{q}_j),z_i)$ denotes the $n \times n$ matrix $\left( \frac{\p^2 }{\p q^{l_1} \p \dot{q}^{l_2} }k((q_j,\dot{q}_j),z_i) \right)_{l_1,l_2=1}^n$ with a corresponding interpretation of $\D^2_{\dot{q},\dot{q}}$. We then build the $n \times K$ matrix
	\[
	R_j = \begin{pmatrix}
		r_{j,1}, \ldots, r_{j,K}
	\end{pmatrix}
	\]
	which are stacked above each other in the vertical direction to form an $nK \times K$-dimensional matrix. Now the system
	\begin{equation}\label{eq:dataconsistencyGPEL}
		R  k(Z,Z)^{-1}L(Z) = 0
	\end{equation}
	corresponds to the condition that \eqref{eq:ELHess} is fulfilled for all data points.
	
	\item
	Finally \eqref{eq:dataconsistencyGPEL} together with the same non-triviality and normalisation condition as employed for learning the inverse modified Lagrangian is solved in the least square sense for $B=k(Z,Z)^{-1}L(Z) \in \R^{nK}$. Notice that a computation of $k(Z,Z)^{-1}$ is not required.

	\item
	The learned Lagrangian $L$ can be evaluated at a point $(q,\dot{q})$ as $L(q,\dot{q}) = k((q,\dot{q}),Z)B$.

\end{enumerate}

For comparison, a neural network version of the above idea would employ a loss function of the form
$\ell = \|\ell_{\mathrm{nontrivial}}\|^2+\ell_{\mathrm{data}}$ with $\ell_{\mathrm{nontrivial}}$ as in \eqref{eq:normaliseLimodQuad} with $L$ instead of $L_\imod$ and 
\[\ell_{\mathrm{data}} = \sum_{i=1}^{K} \sum_{k=1}^{n}\left\| \frac{\p L}{ \p \dot{q}^k}(q_i,\dot{q}_i) - \sum_{j=0}^n \left( \frac{\p^2 L}{ \p q^j \p \dot{q}^k}(q_i,\dot{q}_i) \dot{q_i}^j - \frac{\p^2 L}{ \p \dot{q}^j \p \dot{q}^k}(q_i,\dot{q}_i) \ddot{q_i}^j\right)\right\|^2,
\]
which corresponds to enforcing the Euler--Lagrange equations to be consistent with the training data.

\begin{remark}[Comparison with LNN]
In contrast, Lagrangian neural networks (LNN), as introduced in \cite{LNN}, use a loss function based on comparing the acceleration data $\ddot{q_i}$ with 
\[
\left(\left(\frac{\p^2 L}{ \p \dot{q}^j \p \dot{q}^k} (q_i,\dot{q}_i) \right)_{j,k=1}^n\right)^{-1} \left( \nabla_{q} L (q_i,\dot{q}_i) - \left( \frac{\p^2 L}{ \p q^j \p \dot{q}^k}(q_i,\dot{q}_i) \right)_{j,k=1}^n \dot{q}_i \right).
\]

In LNN the implicit assumption that $\left(\frac{\p^2 L}{ \p \dot{q}^j \p \dot{q}^k} (q_i,\dot{q}_i) \right)_{j,k=1}^n$ is invertible, informs the network that the Lagrangian is regular. The regularity assumption prohibits the LNN to learn the constant solution $L\equiv \mathrm{const}$ which would be consistent with the data but does not generate any dynamics. In contrast, our approach avoids the matrix inversion and imposes non-triviality by a term which normalises the symplectic volume of a unit hypercube in the space with coordinates $(q,\dot{q})$. We can, therefore, control the scaling of the learned Lagrangian and the corresponding symplectic structure and Hamiltonian.
\end{remark}

\section{Numerical Experiments}\label{sec:ex}

In the following, we will demonstrate the performance of Lagrangian Shadow Integration (LSI) on position data of the mathematical pendulum and of the H{\'e}non--Heiles system.
Within LSI we will use the second order accurate variational midpoint discretisation introduced in \cref{ex:MPandTZ} as a base method.
All integrations required to compute the training data set are performed by applying the St{\"o}rmer-Verlet scheme with a very small step-size to the exact equations of motion. After integration, we retain only position data, since LSI uses position data only and does not require observations of velocities, accelerations, or prior knowledge about conjugate momenta.
Based on the training data, the inverse modified Lagrangian is modelled as a Gaussian process and learned as described in \cref{subsec:GPLearnimodL}.

For comparison, we employ Lagrangian Gaussian Processes (LGPs) (see \cref{sec:LGP}) whose training targets are the exact Lagrangian. The learned Lagrangians are then integrated using the variational midpoint rule with the same step-size as we use for LSI. Velocity and acceleration data for training LGP are interpolated using central finite differences. The approximations are second order accurate in the discretisation parameter $h$. Notice that this reduces the available training data since no second order accurate acceleration data is available at the first two and last two snapshots of each trajectory. 

All Gaussian Processes (GP) employed in this section use scaled radial basis functions
\begin{equation}
\label{eqs:rbf}
k(x,y)= c_k \exp \left( - \frac {\|x-y\|^2} {\e^2} \right)
\end{equation}
with parameters $\e$ related to the typical length scale and $c_k$ acting as a scaling factor.

\subsection{Pendulum}

The mathematical pendulum is the dynamical system defined by the second order ordinary differential equations $\ddot{q}=-\sin(q)$, which is the Euler--Lagrange equation to $L_\reff(q,\dot{q}) = \frac 12 \dot{q}^2 + \cos(q)$. Motions preserve the energy $H_\reff(q,\dot{q}) = \frac 12 \dot{q}^2 - \cos(q)$. We consider the step-size $h=0.5$. 
To generate training data, we compute $N=400$ trajectories of length $N_l=6$ with initial values $(q,\dot{q})$ obtained from a Halton sequence on the domain $\Omega = [-\pi,\pi] \times [-1.2,1.2]$. Only the position data in $[-\pi,\pi]$ is retained.

In the following, the parameters of the radial basis functions \eqref{eqs:rbf} of the GPs within LSI and LGP are set to $\e =5$ and $c_k=1$.
\begin{figure}
	\centering
\includegraphics[width=0.45\textwidth]{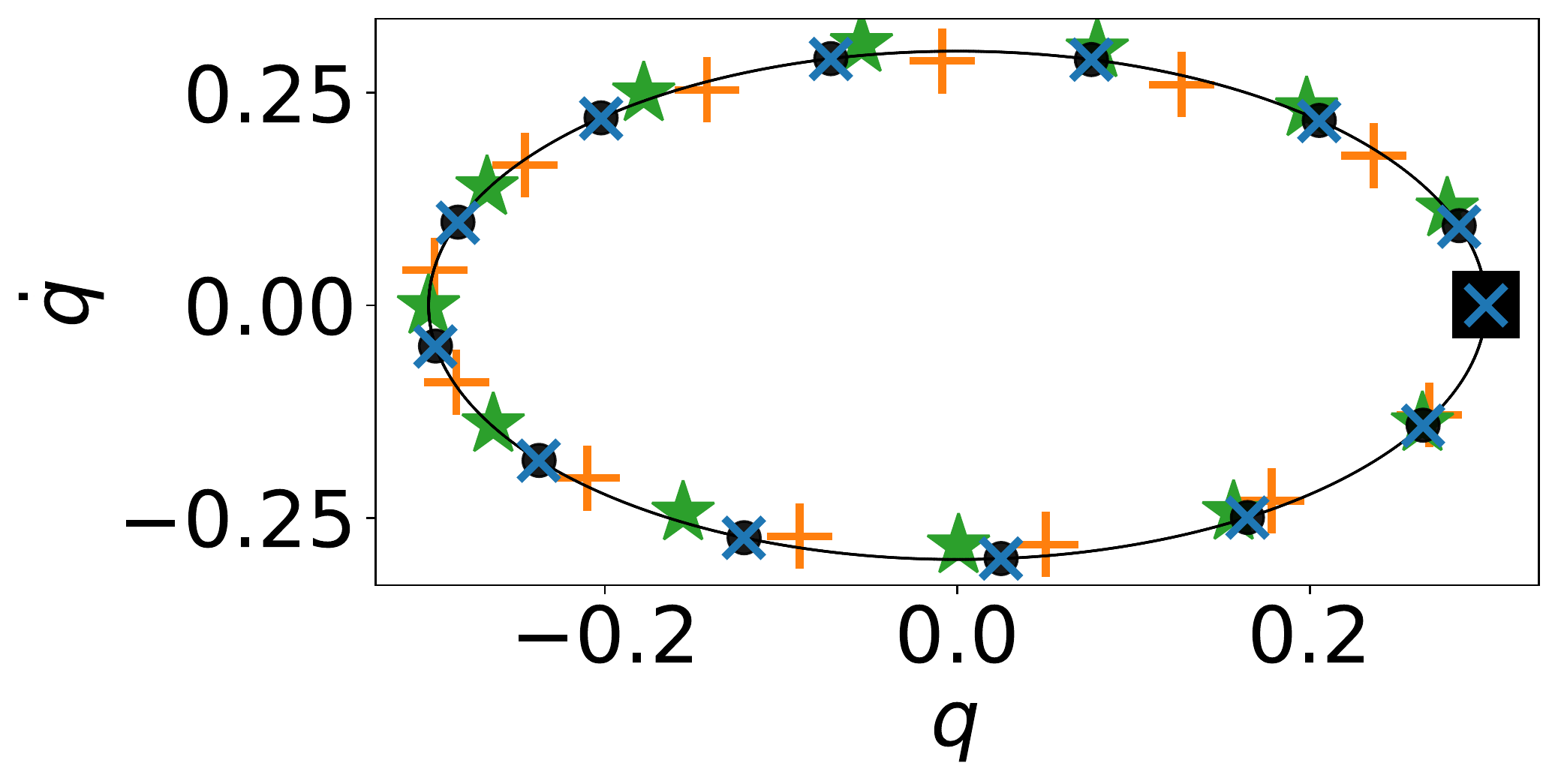}
\includegraphics[width=0.45\textwidth]{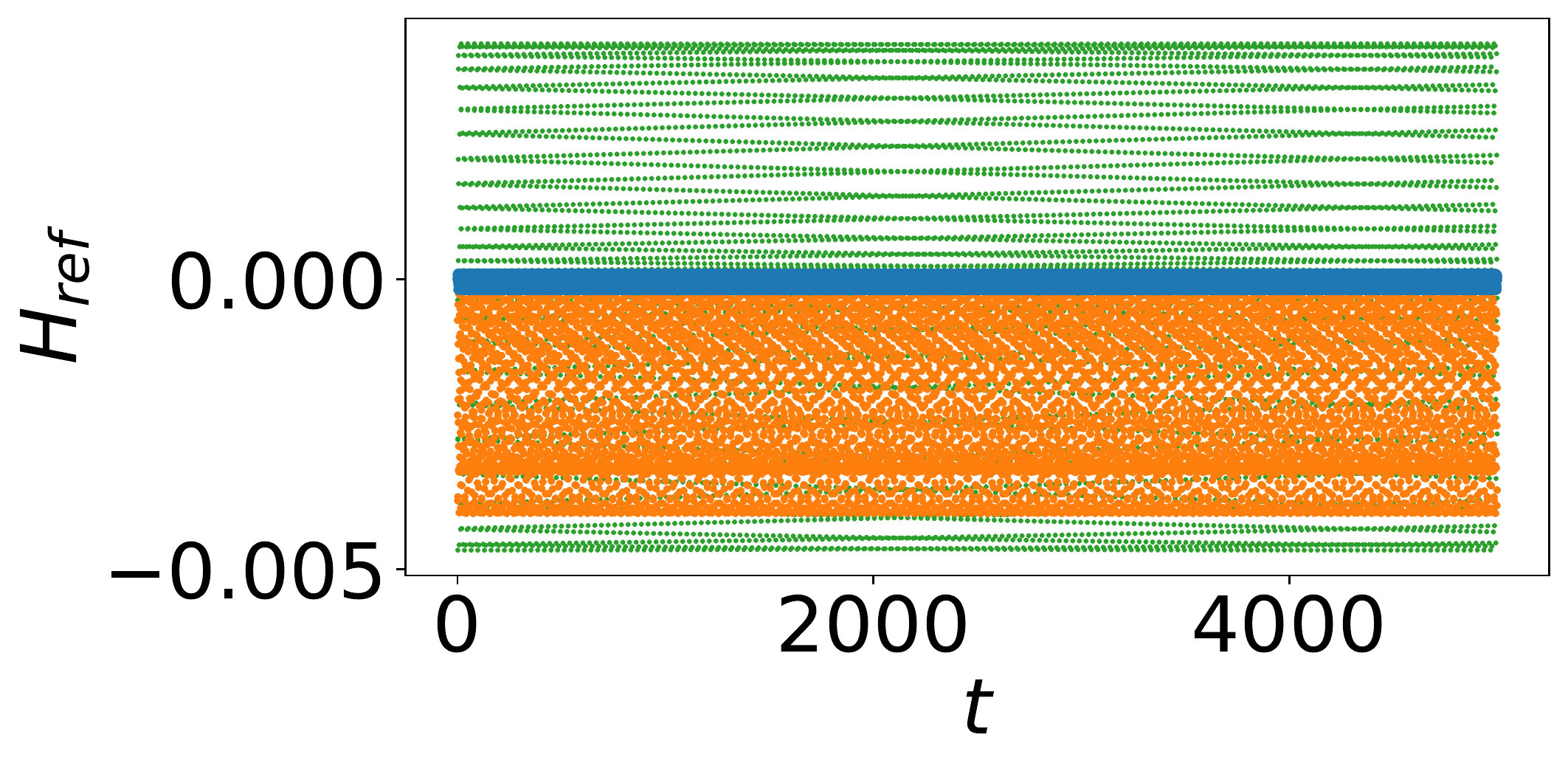}
\caption{Left: The first 13 snapshots to the step-size $h=0.5$ of a trajectory initialised at $(q,\dot{q})=(0.3,0)$ ($\blacksquare$) computed with LSI ($\color{blue}\times$), with LGP ($\color{orange}+$), and with LGPExact ($\color{green}\ast$) are shown. A reference solution is given in solid black, with position and velocity at the snapshot times marked by $\bullet$. Right: Energy behaviour of LSI (blue), LGP (orange), and LGPExact (green) (subsampled). }\label{fig:PendulumSnapshot}
\end{figure}
\Cref{fig:PendulumSnapshot} shows snapshot data of a trajectory computed with LSI as well as with LGP, whereas for LGP all required velocity and acceleration data was approximated using central finite differences. For additional comparison, we plot snapshot data of a trajectory computed with LGPExact. LGPExact is a Lagrangian Gaussian Process as well. However, it is trained with exact velocity and acceleration data. Although this violates our assumption that only position data is available, the experiment is included as it helps us to analyse the effects of discretisation errors in LGP. For training of LGPExact we used the same position data as for LSI. Acceleration data was computed from the differential equation $\ddot q = - \sin(q)$ of the pendulum. Since the right hand side of the differential equation does not explicitly depend on $\dot q$, velocity data was obtained as a Halton sequence of length $N \cdot N_l =2400$ on the interval $(-1.2,1.2)$.

While all methods preserve Lagrangian structure, predictions by LGPExact contain the discretisation error induced by the variational midpoint rule, which is considerable at $h=0.5$ even though the scheme is second order accurate. Additionally, LGP contains another discretisation error, since its training data was obtained using central finite differences. LSI is designed to compensate both discretisation errors and successfully matches the snapshot data of the reference solution.
 
Further, we plot in \cref{fig:PendulumSnapshot} the long-term energy behaviour of the numerical trajectories, i.e.\ we evaluate $H_\reff$ on the computed snapshot data.
All methods show oscillatory energy error behaviour, which is expected, since all methods preserve the variational structure.
We observe that the oscillations of LSI are smaller than those of LGP and LGPExact. This is expected since only LSI contains no inherent discretisation error. Interestingly, a comparison of the energy errors of LGPExact and LGP suggests that in LGP discretisation errors from training and the application of the variational midpoint rule compensate one another partially.

To analyse the quality of predictions obtained by LSI, it is convenient to pass to the Hamiltonian setting due to the non-uniqueness of Lagrangians. We will, therefore, focus on the modified energy $H^{[[k]]}$ preserved by the scheme. Up to errors due to limited training data and truncation errors, we expect the modified energy to have the same level sets as $H_\reff$. To compute $H^{[[k]]}$, first $L^{[[k]]}$ is computed from the learned inverse modified Lagrangian $L_\imod$ using the backward error analysis formula \eqref{eq:MPBEA}, where $k$ denotes the truncation index (also see \cref{thm:VarBEA}). The modified Hamiltonian $H^{[[k]]}$ is then computed using \eqref{eq:HamIntro} applied to $L^{[[k]]}$ (instead of $L$).
\begin{figure}
	\centering
\includegraphics[width=0.9\textwidth]{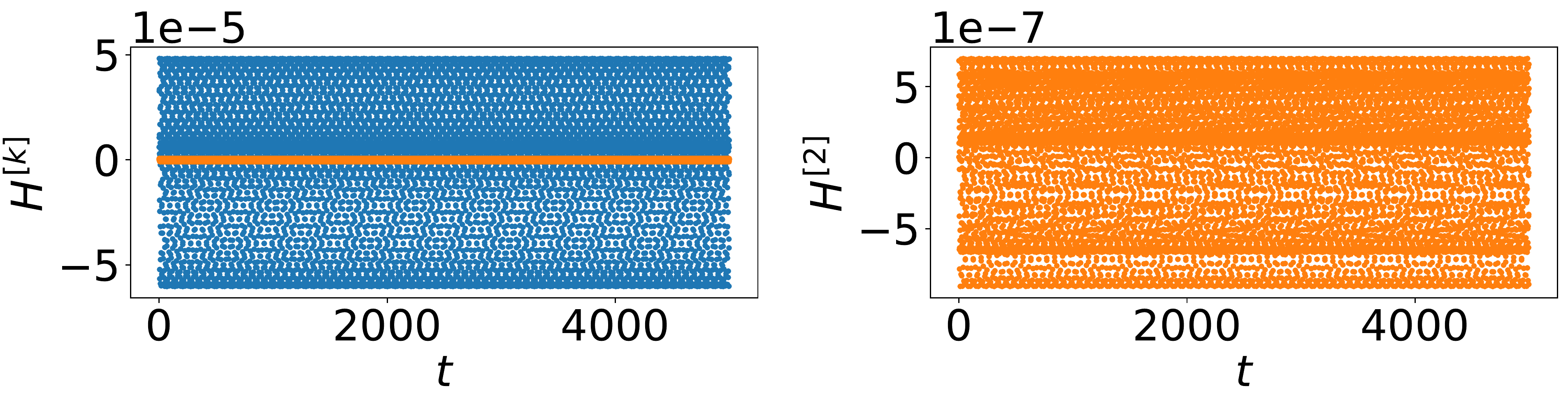}
\caption{$H^{[[0]]}$ (blue) and $H^{[[2]]}$ (orange) are up to an oscillation within a band of width $\approx 10^{-4}$ and $\approx 10^{-6}$, respectively, conserved quantities of the numerical motion predicted by LSI initialised at $(q,\dot{q})=(0.3,0)$. $H^{[2]}$ can, therefore, be used to perform system identification or to verify LSI.}\label{fig:PendulumModEnergyLSIMotion}
\end{figure}
\Cref{fig:PendulumModEnergyLSIMotion} shows that $H^{[[2]]}$ governs the numerical motions well as it is preserved up to an oscillation within a band of width $\approx 10^{-6}$, which is small enough for our purposes. We can, therefore, use $H^{[[2]]}$ to analyse qualitative aspects of numerical motions predicted by LSI: up to truncation error, for a modified symplectic structure $\omega^{[[2]]} = \frac{\p L^{[[2]]}}{\p \dot{q}} \wedge q$ the function $H^{[[2]]}$ is the Hamiltonian which governs the numerical motion. 
\begin{figure}
\includegraphics[width=0.3\textwidth]{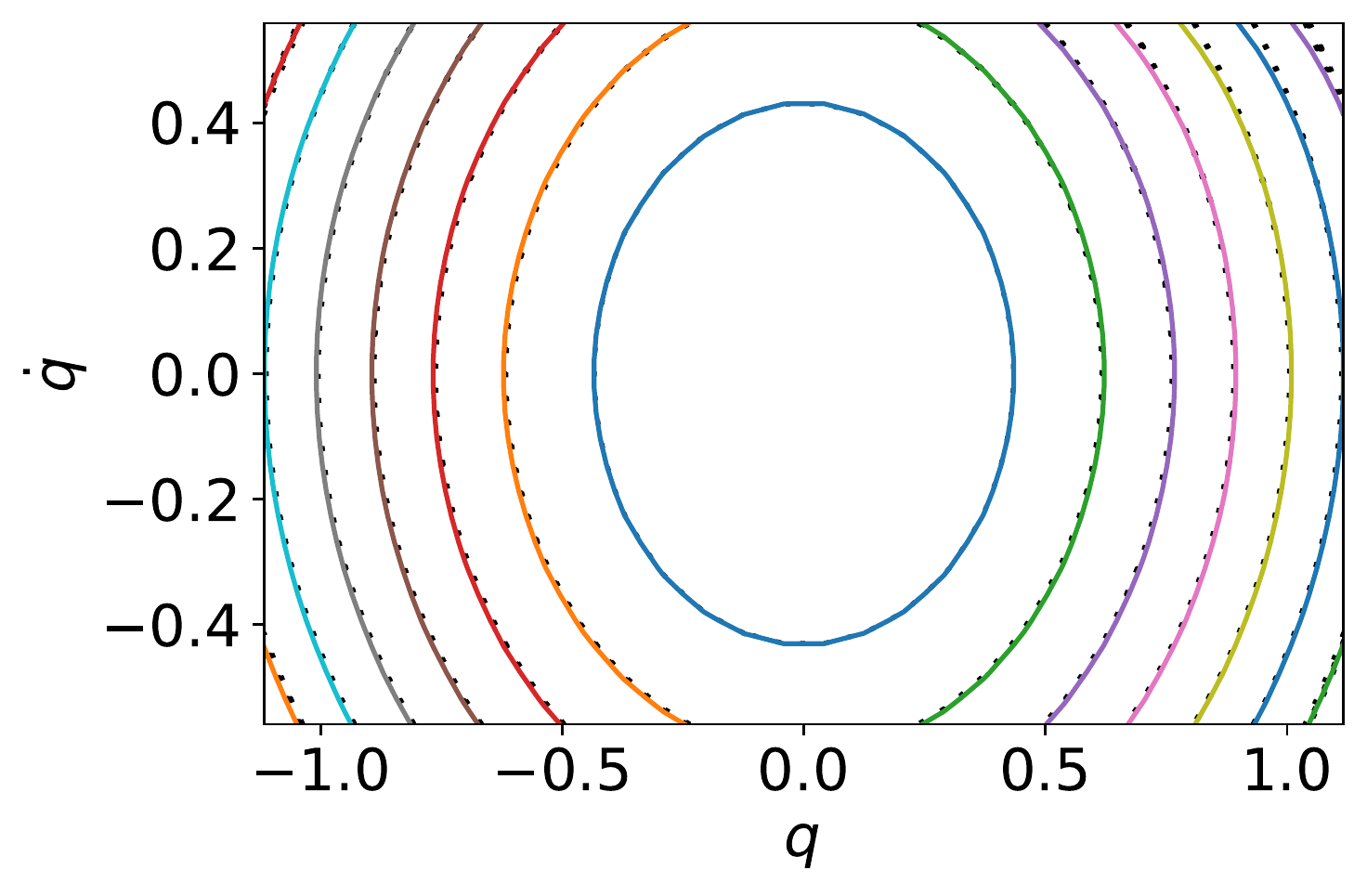}
\includegraphics[width=0.3\textwidth]{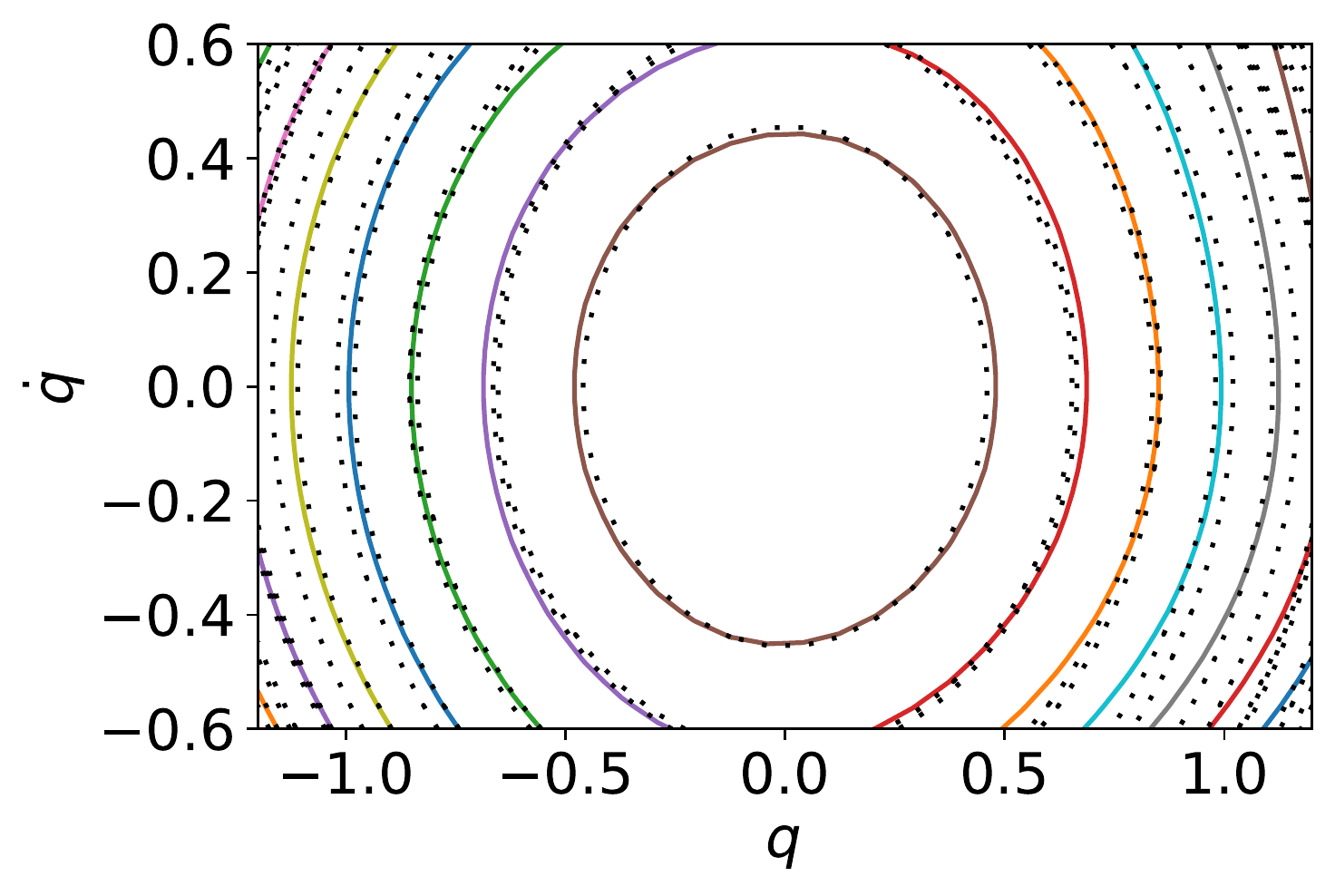}
\includegraphics[width=0.3\textwidth]{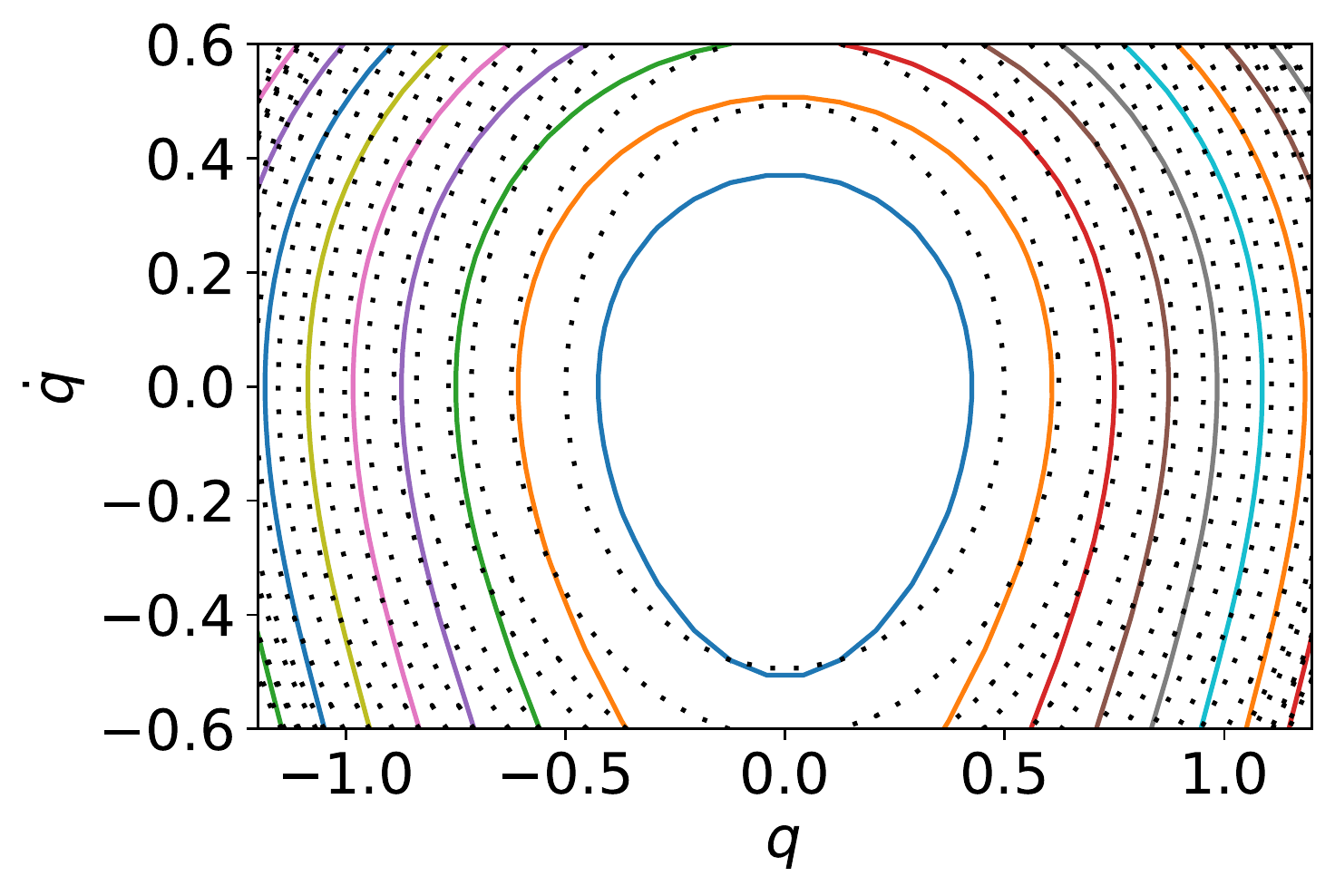}
\caption{Contour plot of the Hamiltonian which governs the numerical motion for LSI (left), LGP (centre), and LGPExact (right). Contours of $H_\reff$, which intersect the solid contour lines, are plotted as dotted lines for reference.}\label{fig:PendulumRecoverPhase}
\end{figure}
The contour plot \cref{fig:PendulumRecoverPhase} shows that the level sets of $H_\reff$ closely match the level sets of $H^{[[2]]}$. Analogous plots for LGP and LGPExact are given for comparison. We see that even in long-term simulations the numerical motions obtained by LSI show qualitatively the correct behaviour.
Motions of LGP and LGPExact, however, will have a pronounced systematic bias.

\begin{figure}
	\includegraphics[width=0.3\textwidth]{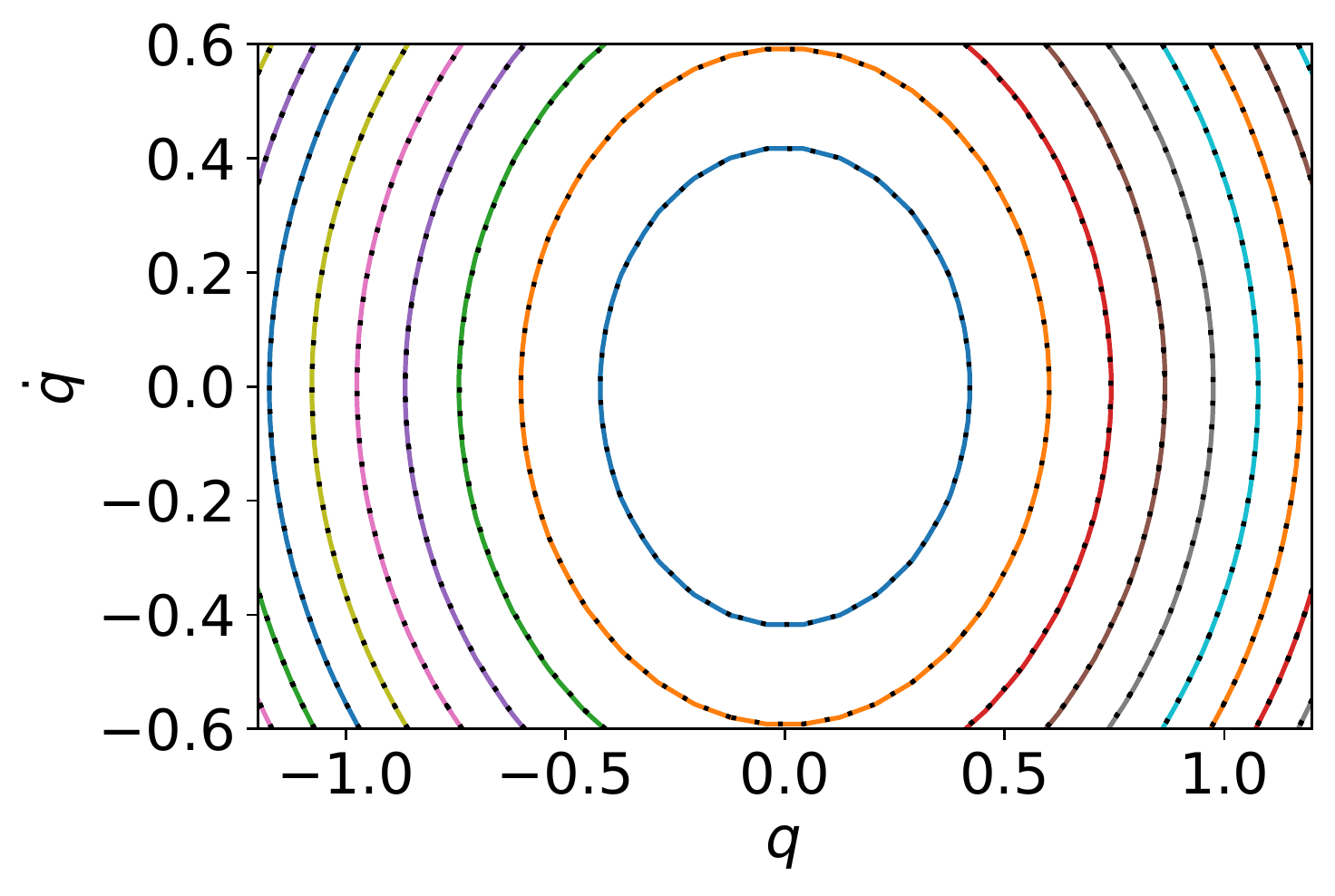}
	\includegraphics[width=0.3\textwidth]{PendulumGPELexBEA.pdf}
	\includegraphics[width=0.3\textwidth]{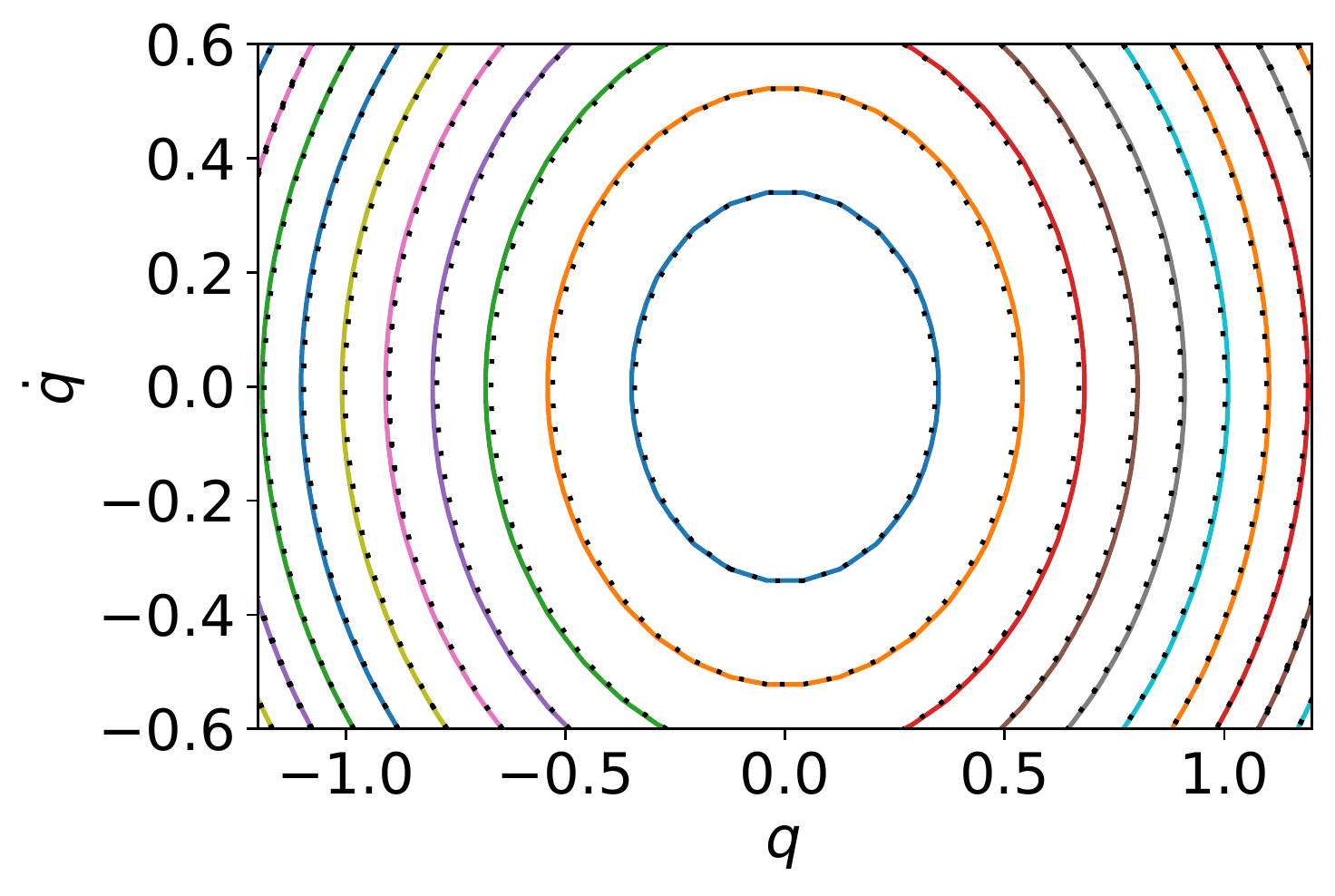}
	\caption{Contour plot of the Hamiltonian identified using LGPExact (left), the Hamiltonian governing the numerical motions when $L_{\mathrm{LGPExact}}$ is integrated (centre), and when the reference Lagrangian $L_{\mathrm{ref}}$ is integrated (right). This shows that although $L_{\mathrm{LGPExact}}$ is a highly accurate Lagrangian for the dynamical system, it is not well suited for numerical integration.}\label{fig:PendulumPhaseBEA}
\end{figure}

\begin{remark}
It might appear surprising that LGPExact performs quite poorly: the Lagrangian $L_{\mathrm{LGPExact}}$ learned by LGPExact is consistent with the dynamical system (\cref{fig:PendulumPhaseBEA} -- left). However, $L_{\mathrm{LGPExact}}$ is not very suitable for numerical integration: the numerical motions follow the contour lines of the centre plot of \cref{fig:PendulumPhaseBEA}. 
For comparison, for the reference Lagrangian $L_\mathrm{ref}$, the numerical motions follow energy level sets that closely match the level sets of $H_\mathrm{ref}$.
This shows that even if highly accurate velocity data and acceleration data is available, Shadow Integration can be worthwhile
as it compensates discretisation errors in the integration step also for Lagrangians that would otherwise not be well suited for numerical purposes.
\end{remark}


Recall that in this context Hamiltonians are defined only up to additive constants and scaling factors by the dynamics since the symplectic structure is not fixed. To quantify the difference between the level sets of $H_\reff$ and $H^{[[2]]}$ we measure how parallel their gradients are. More precisely, we measure the area $\mu(q,\dot{q})$ of the parallelograms spanned by the normalised gradients $\nabla H_\reff/\|\nabla H_\reff\|$ and $\nabla H^{[[2]]}/\| H^{[[2]]}\|$ at $(q,\dot{q})$. In this 2-dimensional example, we have
\[
\mu(q,\dot{q})=\left|\det\left( \frac{\nabla H_\reff(q,\dot{q})}{\|\nabla H_\reff(q,\dot{q})\|},\frac{\nabla H^{[[2]]}(q,\dot{q})}{\| \nabla H^{[[2]]}(q,\dot{q})\|} \right)\right|.
\]
We sum $\mu$ over an equidistant mesh with $30 \times 30$ data points on the domain $\Omega_0 = [-1.2,1.2]\times[-0.6,0.6]$ and obtain after division by the number of data points $\nu_{\mathrm{LSI}}^{\mathrm{BEA}} \approx 0.01$. For LGP we obtain $\nu_{\mathrm{LGP}}^{\mathrm{BEA}} \approx 0.05$ and $\nu_{\mathrm{LGPExact}}^{\mathrm{BEA}} \approx 0.1$ for LGPExact. It is interesting to observe that LGPExact performs significantly worse than LGP. Again, this suggests that some of the inherent discretisation errors of LGP compensate each other but to a lesser extend than in the systematic approach of the LSI framework.

When system identification is the goal rather than an analyis of the numerical motions obtained by applying the variational midpoint rule to a Lagrangian, for LGP and LGPExact we can compute $H_{\mathrm{LGP}}$ and $H_{\mathrm{LGPExact}}$ by applying \eqref{eq:HamIntro} to the learned Lagrangians $L_{\mathrm{LGP}}$ and $L_{\mathrm{LGPExact}}$ directly rather than applying the backward error analysis formula \eqref{eq:MPBEA} first (which is necessary when LSI is used). In this context LGP only suffers from a discretisation error in the training data, while in the experiment of \cref{fig:PendulumSnapshot} an additional discretisation error of the variational midpoint rule is present. We get $\nu_{\mathrm{LGP}} \approx 0.03$. Compared to $\nu_{\mathrm{LSI}}=\nu_{\mathrm{LSI}}^{\mathrm{BEA}} \approx 0.01$ this shows that LSI has an edge over LGP when it comes to system identification. Moreover, we can conclude that even if $L_{\mathrm{LGP}}$ was integrated with tiny time-steps in the motion prediction test, solutions would improve but cannot reach the quality of LSI predictions.

Besides, LGPExact does not suffer from any discretisation error in this test and we obtain $\nu_{\mathrm{LGPExact}} \approx 3.4 \cdot 10^{-5}$, which is by magnitudes better than the value for LGP and LSI. However, LGPExact was trained with exact velocity and acceleration data, which might not be available.

\subsection{H{\'e}non-Heiles} The conservative motion of a particle of unit mass subject to the potential
\[
V(q) = \frac 12 \|q\|^2 + \alpha \left( q_1^2 q_2 - \frac{q_2^3}{3} \right)
\]
is governed by the Euler--Lagrange equations \eqref{eq:ELIntro} for the Lagrangian $L(q,\dot{q}) = \frac 12 \| \dot{q}\|^2 - V(q)$.
The energy is given by $H(q,\dot{q}) = \frac 12 \|\dot{q}\|^2+V(q)$ and is conserved along motions. The potential $V$ has two critical values $c_0=0$ and $c_1=\frac 1 {6 \alpha^2}$. The set $V^{-1}(c_1)$ separates the phase space into six unbounded components and one bounded component $A$ containing the origin \cite{Henon1964}.
(For a contour plot of $V$ see \cref{fig:HenonPotential}.)

In the following numerical experiments we used $\alpha=0.8$ and the step-size $h=0.1$. Training was performed on $N=200$ trajectories of length $N_l=5$. The initial values of the trajectories were obtained from a Halton sequence on the domain $\Omega = [-0.8,0.8]^4$. Again, velocity data was discarded and only the position data of the training set is used in the following.
\begin{figure}
\includegraphics[width=0.3\textwidth]{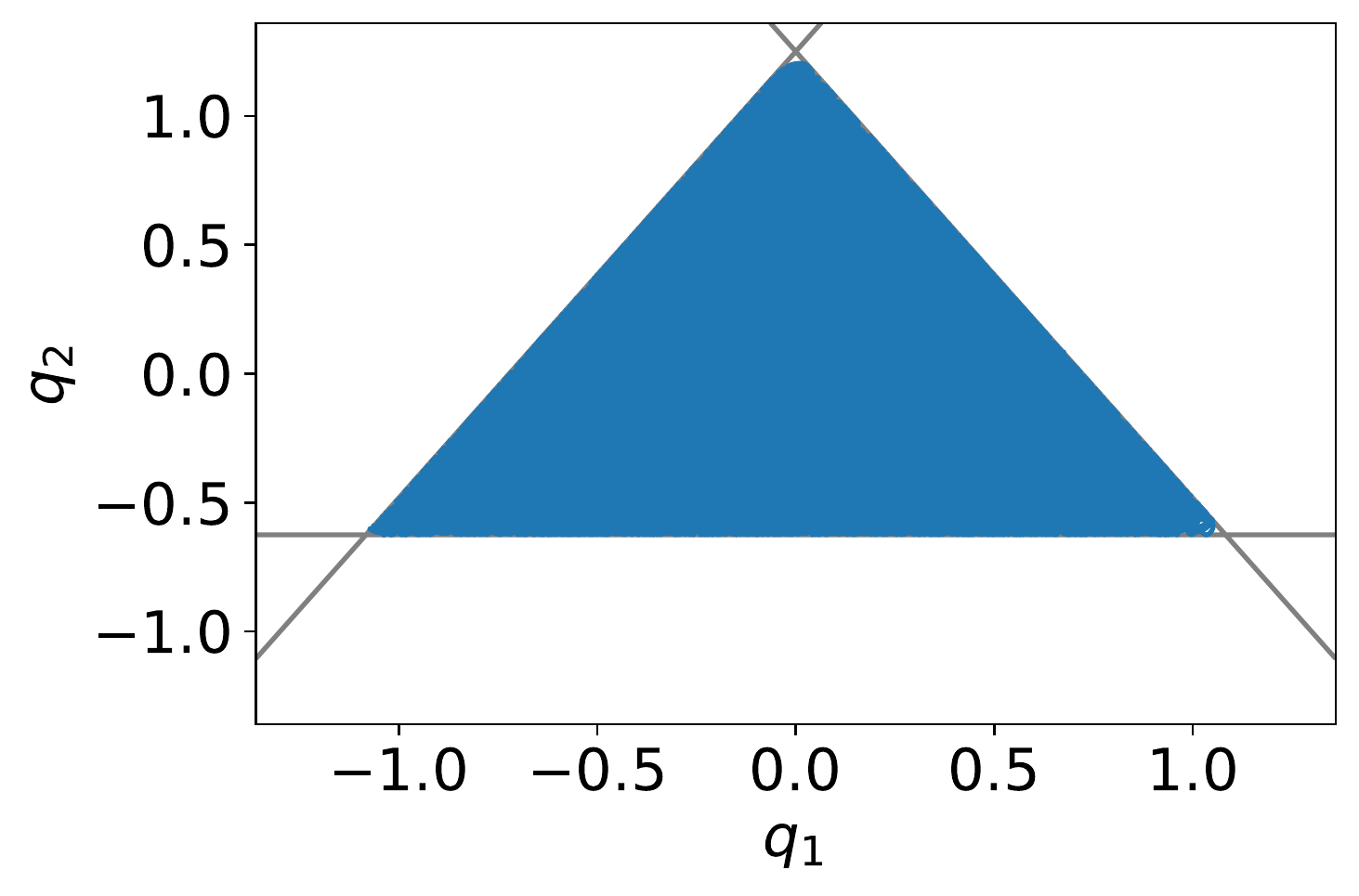}
\includegraphics[width=0.3\textwidth]{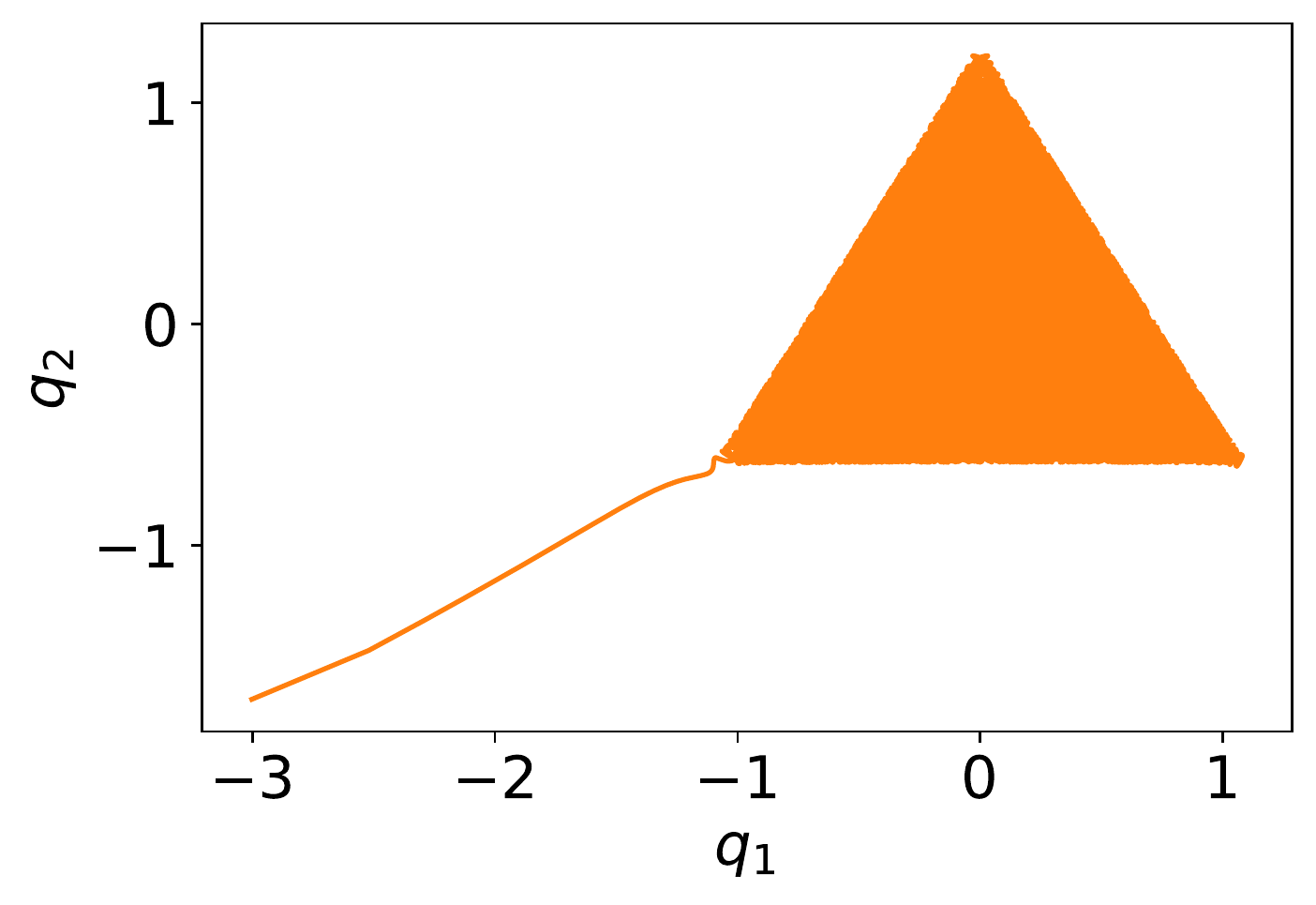}
\includegraphics[width=0.3\textwidth]{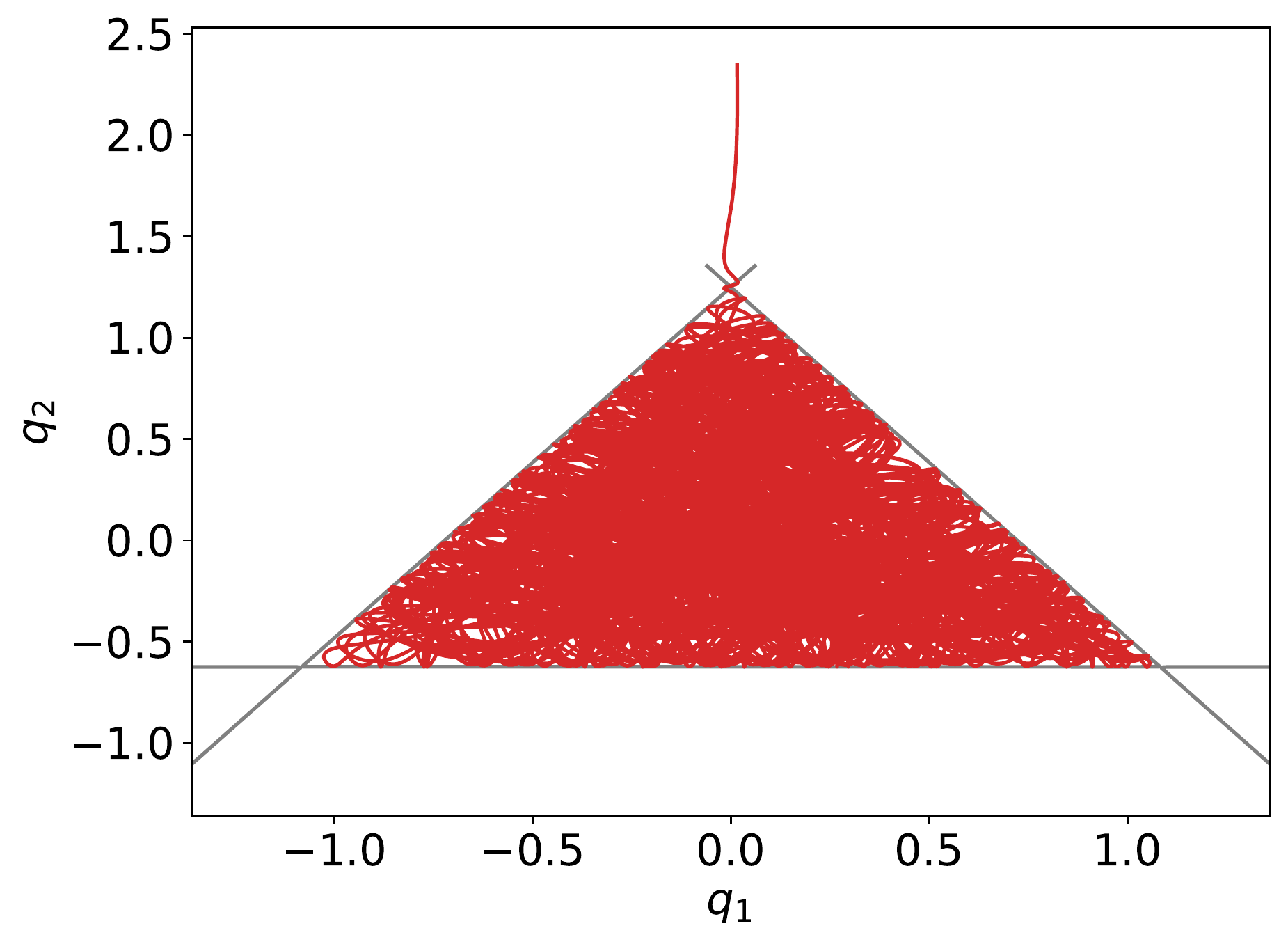}
\caption{Trajectory initialised at $(q_0,\dot{q}_0)=(0.675499,0.08,0,0)$ predicted by LSI until $T=1.774 \cdot 10^4$ (left), by LGP until $T=7.069 \cdot 10^3$ (centre), and by a GP directly fitted to approximated flow data until $T=1.4574 \cdot 10^3$ (GPFlow). All but the LSI prediction stay in the correct region of the phase space until time $T=1.774 \cdot 10^4$.}\label{fig:HenonTrjs}
\end{figure}
\Cref{fig:HenonTrjs} shows predictions of trajectories computed by LSI and by LGP. For the underlying GPs we used radial basis functions \eqref{eqs:rbf} with parameters $c_k=1$ and $\e = 10$. For LGP missing velocity and acceleration data was approximated using second order accurate central finite differences. For further comparison, we compute flow map data by adding velocity data to the position data from the training data set with central finite difference approximations.
Then a GP is fitted to the flow map data using scaled radial basis functions in combination with hyperparameter fitting employing the Python package {\tt scikit-learn}. This method does not incorporate geometric structure and will be referred to by GPFlow.
The trajectories in \cref{fig:HenonTrjs} are initialised at $(q_0,\dot{q}_0)=(0.675499,0.08,0,0)$. Their energy $H(q_0,\dot{q}_0) = V(q_0) \approx 0.26041605$ is close to the critical value of $c_1 = 0.260416\overline{6}$ such that high energy accuracy is required in numerical simulations to prevent erroneous divergence.
While GPFlow, LGP, and LSI all diverge eventually, the LSI prediction shows the correct behaviour (an orbit densely filling the area bounded by $V^{-1}(q_0)$) for much longer than LGP or GPFlow.

\begin{figure}
	\centering
\includegraphics[width=0.8\textwidth]{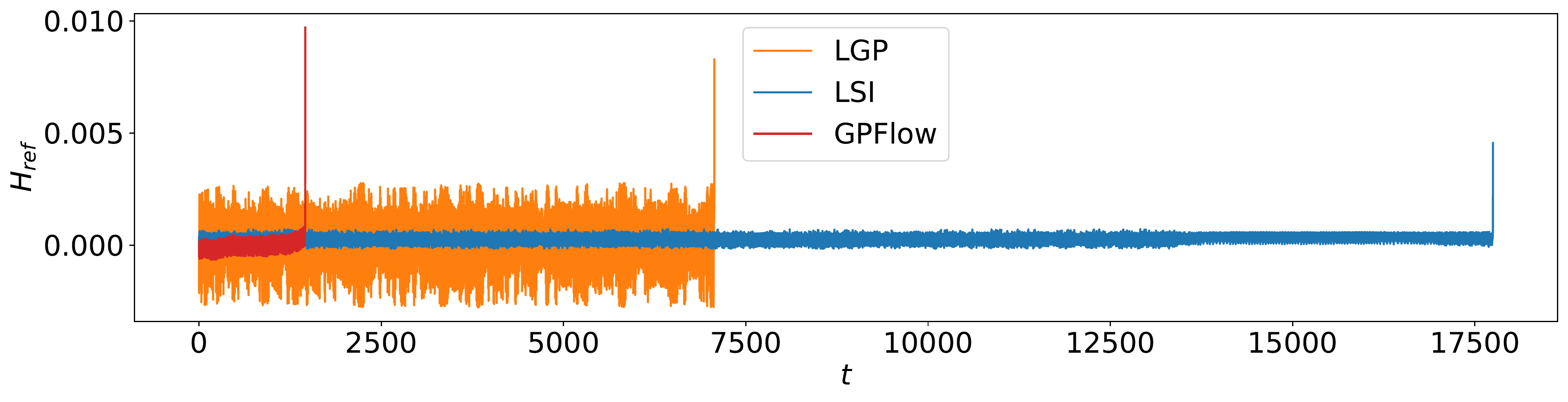}
\caption{Energy behaviour of the trajectories from \cref{fig:HenonTrjs}. All trajectories diverge eventually but the LSI prediction diverges much later than GPFlow and LGP due to its better energy preservation properties.}\label{fig:HenonEnergy}
\end{figure}
The energy error plot (\cref{fig:HenonEnergy}) confirms that LSI preserves energy most accurately, while the energy behaviour of GPFlow shows a drift as no Lagrangian or Hamiltonian structure has been enforced. The steady energy growth causes the trajectory to enter a section of the phase space where trajectories diverge. Although LGP does not suffer from a steady energy growth since variational structure is enforced, its energy error oscillations are wider than those of LSI because discretisation errors in the training process and by the application of variational midpoint integration do not get corrected. The wider oscillations cause the trajectory to diverge erroneously much earlier than the LSI prediction which conserves energy more accurately.

We again use backward error analysis and compute the modified system whose exact motion (up to truncation error) coincide at the snapshot times with the numerical predictions by LSI. As in the previous experiment, we apply backward error analysis to the learned inverse modified Lagrangian $L_\imod$ (truncating after second order terms) and then calculate $H^{[[2]]}_{\mathrm{LSI}}$. 
\begin{figure}
	\centering
\includegraphics[width=0.4\textwidth]{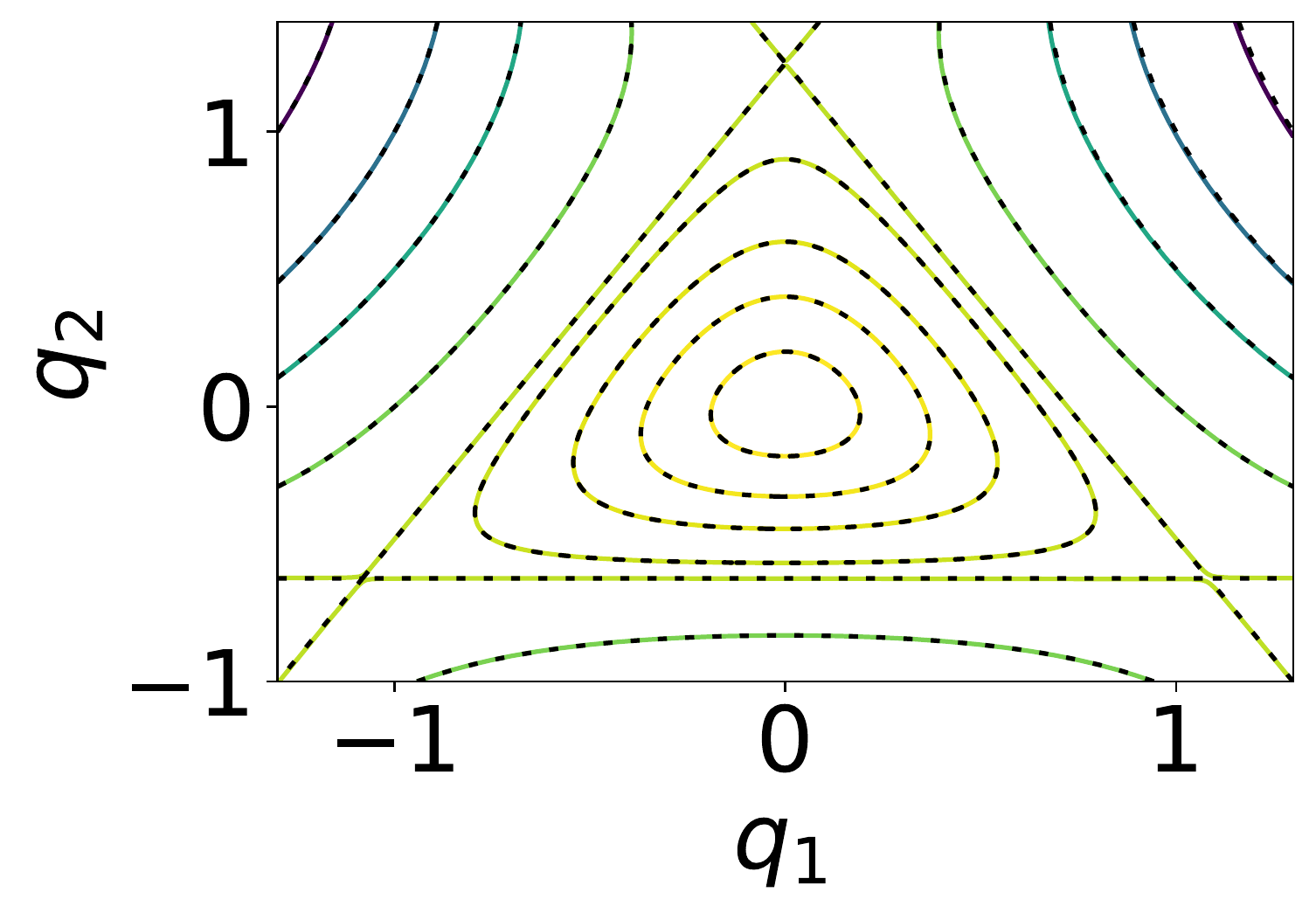}
\includegraphics[width=0.4\textwidth]{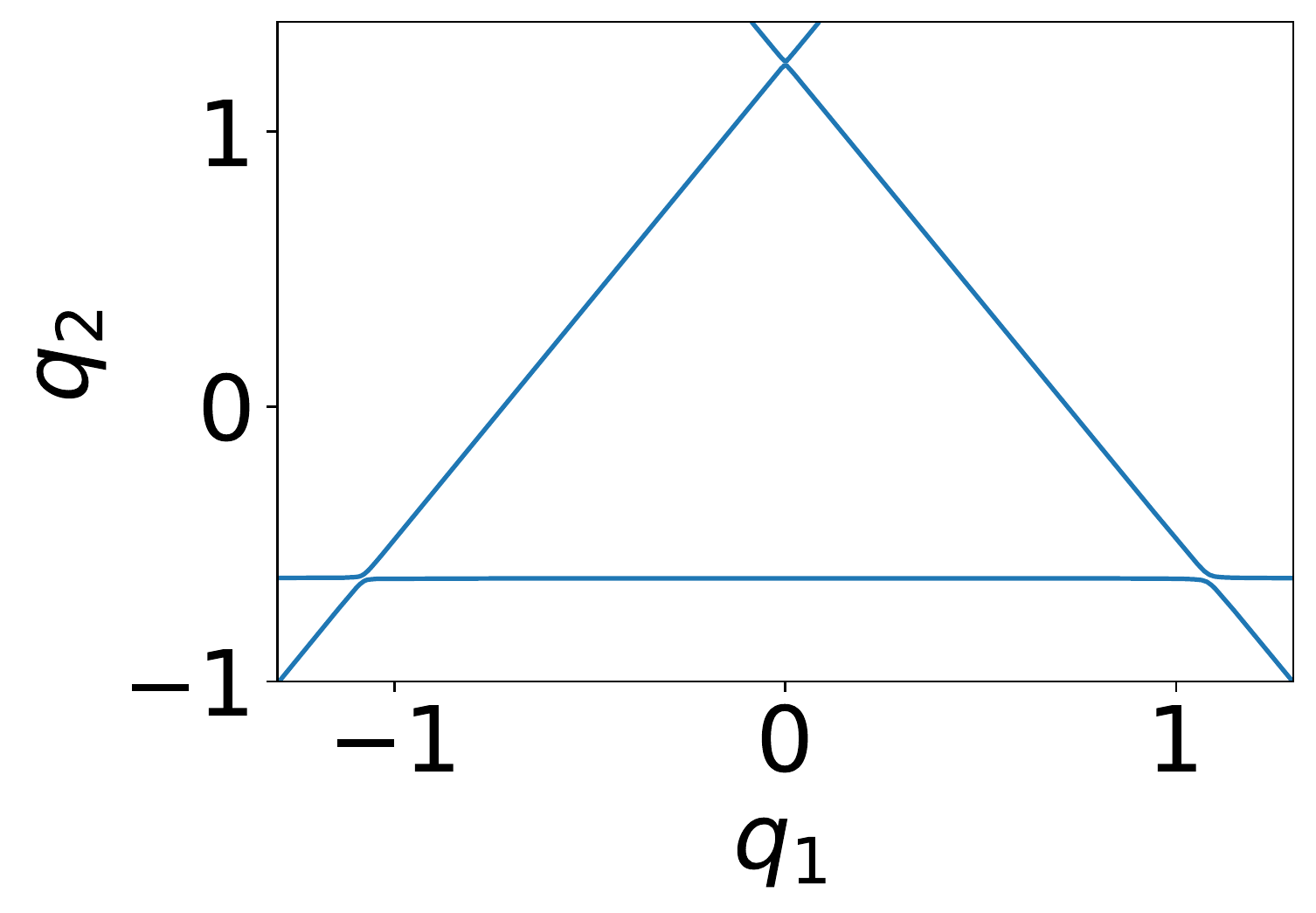}\\
\includegraphics[width=0.32\textwidth]{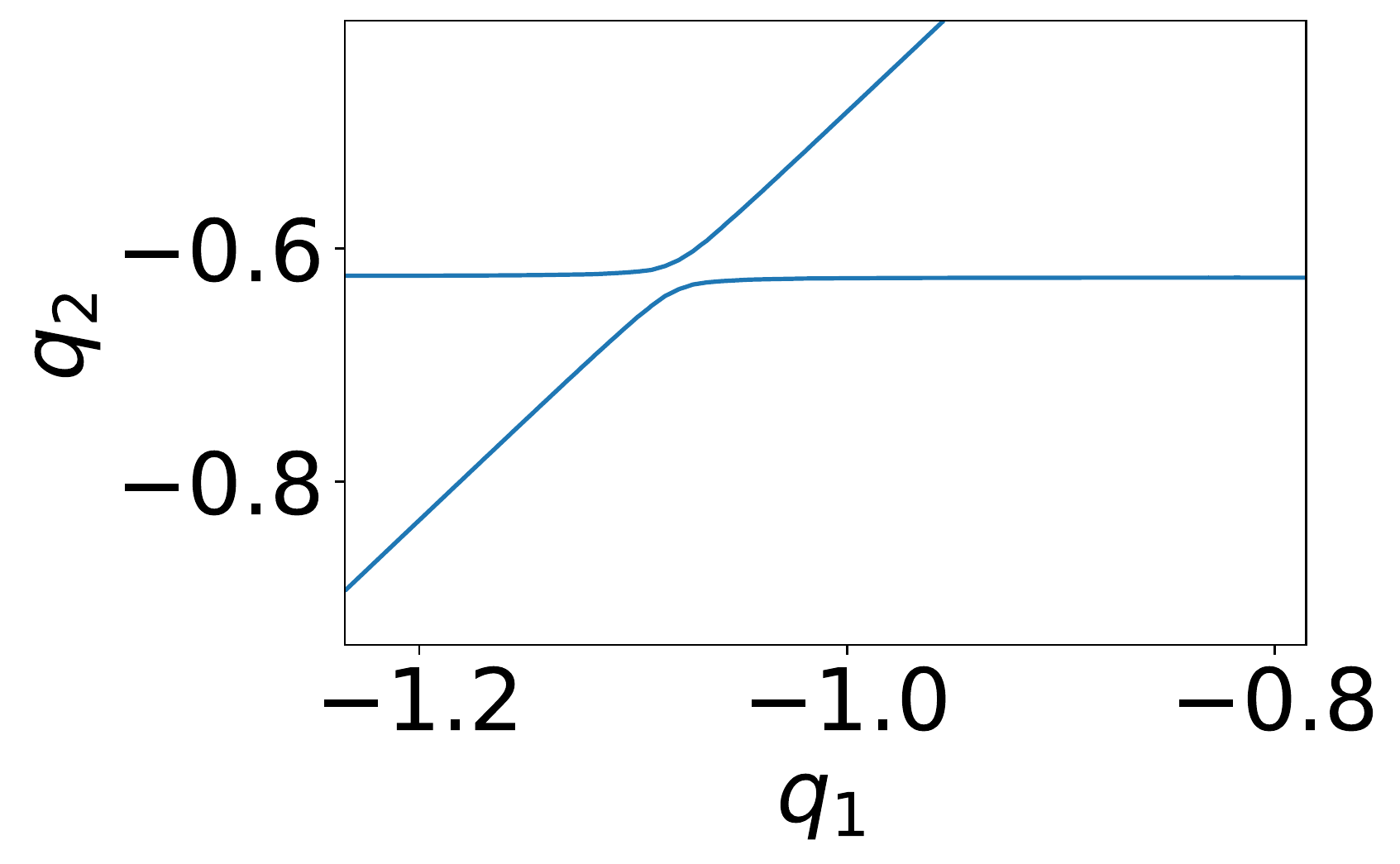}
\includegraphics[width=0.32\textwidth]{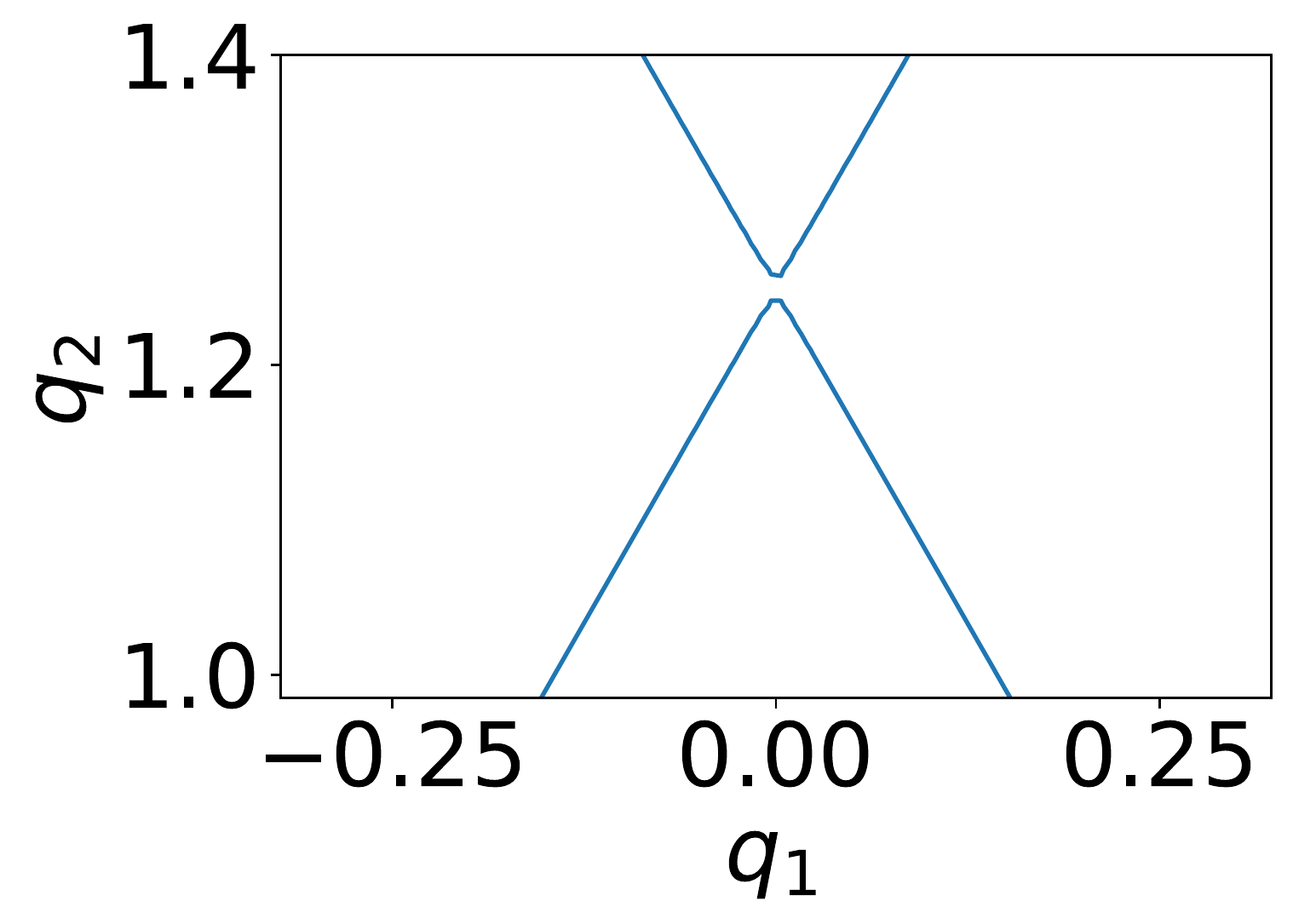}
\includegraphics[width=0.32\textwidth]{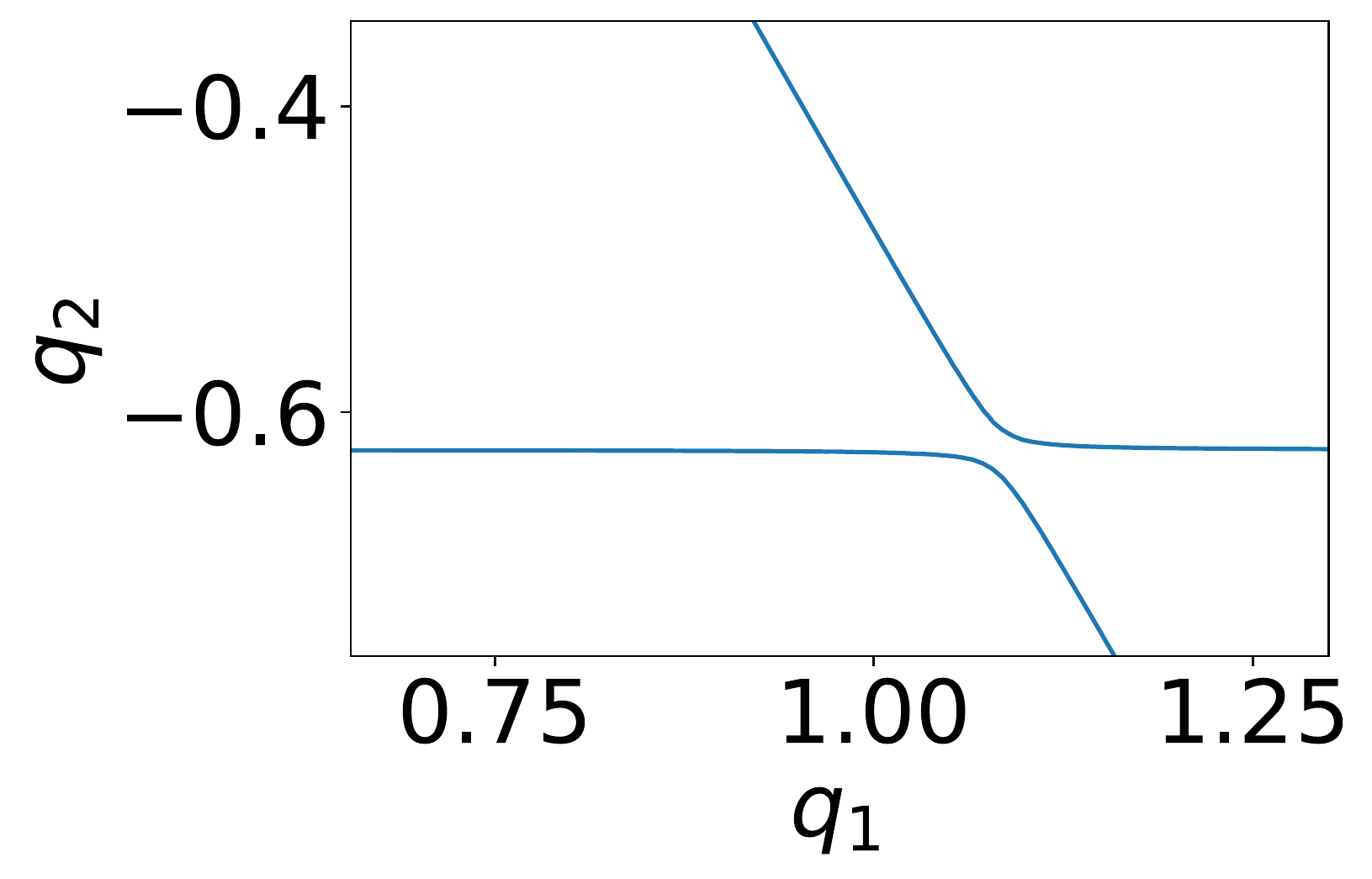}
\caption{The contour lines of the numerical energy (solid lines) closely approximate the contour lines of the energy of the exact system (black dashed lines). This is visualised by viewing the section of the phase space where $\dot{q}=0$ (top left). However, some level sets close to the separatrix of the exact system are incorrectly connected (top right). Close-up plots of the top right figure are displayed in the second row.}\label{fig:HenonPotential}
\end{figure}
A contour plot of $V^{[[2]]}_{\mathrm{LSI}}(q) = H^{[[2]]}_{\mathrm{LSI}} (q,(0,0))$ approximating the contours of the exact potential $V_\reff$ are plotted in \cref{fig:HenonPotential}. The plots confirm that LSI can successfully be used for system identification. Moreover, it verifies that the learned system is close to the exact system and that all contours away from the separatrix have the correct topology. This shows that motions initialised not too close to the separatrix qualitatively show the correct long-term behaviour. However, very close to the separatrix there are incorrectly connected contour lines leaving small channels for the numerical motions with close to critical energy to escape erroneously into another region of the phase space.
This explains why the LSI prediction of the experiment shown in the previous figures diverges eventually.

\section{Inverse modified Lagrangians}\label{sec:Existence}

In this section, we review variational backward error analysis which was developed in \cite{Vermeeren2017}.
Its methods are needed for the system identification part of Lagrangian Shadow Integration. Moreover, we develop an inverse version of variational backward error analysis and prove the existence of inverse modified Lagrangians in the setting of formal power series. This provides a theoretical justification of why Lagrangian Shadow Integration works.
 
Let us fix notation and review and extend classical notions in Lagrangian mechanics and in the theory of variational integrators to formal power series. 
In the following, $Q$ denotes an analytic manifold with tangent bundle $TQ$ and $\mathcal{C}^a(Q,\R)$ denotes the linear space of analytic maps from $Q$ to the real line $\R$. The space of $k$-jets over $Q$ is denoted by $\mathrm{Jet}^k(Q)$ for $k \in \N \cup \{ \infty\}$.
We will make use of a local identifications of $Q$ with a subset of $\R^n$ and corresponding trivialisations of $TQ$ and the jet-spaces. Local coordinates on $Q$ are denoted by $q = (q_1,\ldots q_n)$ and coordinates on the jet space by $q^{[k]} = (q, \dot{q},\ddot{q},q^{(3)},q^{(4)},\ldots)$.

\begin{definition}[Euler-Lagrange operator]
	We define the linear operator $\mathcal{E}$ by
	\[
	\mathcal{E}(L) = \sum_{j=0}^k (-1)^j \frac{\d^j}{\d t^j} \frac{\p L}{ \p q^{(j)} }
	\]
	for $L \in \mathcal{C}^a(\mathrm{Jet}^k(Q),\R)$.
	Here $\frac{\d^j}{\d t^j}$ denotes the total derivative. The equation $\mathcal{E}(L) =0$ is called {\em Euler-Lagrange equation} and $\mathcal{E}$ {\em Euler-Lagrange operator}. The definition of $\mathcal{E}$ can be linearly extended to formal power series $L = \sum_{j=0}^\infty s^j L_j \in \mathcal{C}^a(\mathrm{Jet}^k(Q),\R)[[s]]$ in a formal variable $s$.
\end{definition}

\begin{definition}[regular Lagrangian]
	A Lagrangian $L \in \mathcal{C}^a(TQ,\R)$ is {\em regular} or {\em non-degenerate}, if $\left(\frac{\p^2 L}{\p \dot{q}^i \p \dot{q}^j}\right)_{i,j}$ is invertible at $h=0$. A Lagrangian $L =  \sum_{j=0}^\infty s^j L_j \in \mathrm{Jet}^\infty(Q)[[s]]$ with $L_0\in \mathcal{C}^a(TQ,\R)$ is {\em regular} or {\em non-degenerate}, if $L_0$ constitutes a regular Lagrangian.
\end{definition}


\begin{proposition}\label{prop:modeq}
	Let $ \sum_{j=0}^\infty s^j L_j \in \mathcal{C}^a(\mathrm{Jet}^\infty(Q),\R)[[s]]$ be a regular Lagrangian, where $L_0 \in \mathcal{C}^a(TQ,\R)$ and $L_j \in \mathcal{C}^a(\mathrm{Jet}^{m_j}(Q),\R)$ for $m_j < \infty$ for all $j$.
	There exists a formal power series $\sum_{j=0}^\infty s^j g_j \in \mathcal{C}^a(TQ,\R)$ such that for any truncation index $k$ all solutions to the second order equation 
	\begin{equation}\label{eq:modeqtrunc}
		\ddot{q} = \sum_{j=0}^k s^j g_j(q, \dot{q})
	\end{equation}
	solve the Euler-Lagrange equation 
	\begin{equation}\label{eq:ELSeriesTruncated}
		\mathcal{E}\left(\sum_{j=0}^k s^j L_j\right)=0
	\end{equation}
	up to an error of order $\mathcal{O}(s^{k+1})$.
	
\end{proposition}

\begin{proof}
	Let $k$ be a truncation index and $m_{\max} = \max(m_0,m_1,\ldots,m_k)$. Since $L_0 \in \mathcal{C}^a(TQ,\R)$ is a regular Lagrangian, the Euler-Lagrange equations \eqref{eq:ELSeriesTruncated} are equivalent to a differential equation of the form
	\begin{equation}\label{eq:premod}
		\ddot{q} = g_0(q,\dot{q}) + \sum_{j=1}^k s^j \bar{g}_j(q, \dot{q},\ddot{q},\ldots,q^{(m_{\max})}).
	\end{equation}
	Time derivatives of \eqref{eq:premod} yield expressions for $\ddot{q},\ldots,q^{(k_{\max})}$.
	Repeated substitutions of these expressions into the right hand side of \eqref{eq:premod} yield after finitely many steps
	\begin{equation}\label{eq:premodred}
		\ddot{q} = \sum_{j=0}^k s^j g_j(q, \dot{q}) + \mathcal{O}(s^{k+1}).
	\end{equation}
	Solutions to \eqref{eq:modeqtrunc} solve \eqref{eq:premodred} and thus the Euler--Lagrange equations \eqref{eq:ELSeriesTruncated} up to an error term of order $\mathcal{O}(s^{k+1})$.
\end{proof}

\begin{definition}[variational integrator]\label{def:variationalintegrator}
	A variational integrator $I$ is a bounded linear map \[I \colon \mathcal{C}^a(TQ,\R) \to  \mathcal{C}^a(Q \times Q \times \R, \R).\qedhere\]
\end{definition}

\begin{remark}
	Images of variational integrators $L_\Delta = I(L) \in \mathcal{C}^a(Q \times Q \times \R, \R)$ are called {\em discrete Lagrangians}. Their last input argument is interpreted as a discretisation parameter or step-size. It is sometimes implicit in the notation $L_\Delta(q_0,q_1) = L_\Delta(q_0,q_1,h)$.
\end{remark}

\begin{definition}[consistent variational discretisation/integrator]
	A discrete Lagrangian $L_\Delta \in \mathcal{C}^a(Q \times Q, \R)$ is a {\em consistent variational discretisation} of a Lagrangian $L \in \mathcal{C}^a(TQ,\R)$ if it has variational order of at least 1, i.e.\ there exists an open set $U \subset TQ$ with compact closure and constants $c,\tau >0$ such that for all solutions of the Euler--Lagrange equations
	\[
	0=\mathcal{E}(L)
	=\frac{\p L}{\p q}(q,\dot{q}) - \frac{\d}{\d t}\frac{\p L}{\p \dot{q}}(q,\dot{q})
	\]
	with initial conditions $(q(0),\dot{q}(0)) \in U$ we have
	\[
	\left\| I(L)(q(0),q(h)) - \int_0^h L (q(t),\dot{q}(t)) \d t \right\| \le c h^2
	\]
	for all $h \le \tau$.
	
	A variational integrator $I$ is {\em consistent} if for all regular Lagrangians $L \in \mathcal{C}^a(TQ,\R)$ the discrete Lagrangian $I(L)$ is a consistent variational discretisation of $L$.
\end{definition}


\begin{definition}[(consistent) variational integrator on power series]
	Variational integrators can be defined on $\mathcal{C}^a(TQ,\R)[[s]]$ by a linear extension of \cref{def:variationalintegrator}. Here $\mathcal{C}^a(TQ,\R)[[s]]$ denotes the ring of formal power series in the formal variable $s$ over $\mathcal{C}^a(TQ,\R)$ (interpreted as a ring).
	A variational integrator $I$ is {\em consistent} if for any $L = \sum_{j=0}^\infty s^j L_j \in \mathcal{C}^a(TQ,\R)[[s]]$ the discrete Lagrangian $I(L_j)$ is a consistent variational discretisation of $L_j$.
\end{definition}


\begin{definition}[modified equation]\label{def:modeq}
	Consider a consistent variational discretisation $L_\Delta = I(L)$ of a regular Lagrangian $L$. An application of \cref{prop:modeq} to the series expansion of $L_\Delta(q(t),q(t+s))$ around $s=0$ yields a formal power series $\sum_{j=0}^\infty s^j g_j$. The equation
	\begin{equation}\label{eq:modEQDef}
		\ddot{q} = \sum_{j=0}^\infty s^j g_j(q, \dot{q})
	\end{equation}
	is called {\em modified equation} for $L_\Delta$.
\end{definition}

\begin{remark}
	Since $L$ in \cref{def:modeq} is regular and the discretisation $L_\Delta = I(L)$ consistent, the zeroth order truncation of the modified equation \eqref{eq:modEQDef} is equivalent to the Euler--Lagrange equation $\mathcal{E}(L)=0$.
\end{remark}

\begin{remark}
	Solutions to optimal truncations of the modified equation \eqref{eq:modEQDef} constitute solutions to the discrete Euler-Lagrange equations
	\begin{equation}\label{eq:DELcontinuous}
		\nabla_2 L_\Delta (q(t-s),q(t)) + \nabla_1 L_\Delta (q(t),q(t+s)) = 0
	\end{equation}
	on exponentially long time intervals up to an exponentially small error in the step size $s$.
	In \eqref{eq:DELcontinuous} $\nabla_1 L_\Delta$ and $\nabla_2 L_\Delta$ denote the derivatives with respect to the first or second input argument, respectively.
	See, for instance, \cite{GeomIntegration,Leimkuhler2005} for a rigorous analysis and optimal truncation techniques.
\end{remark}

The following result was developed by Vermeeren in \cite{Vermeeren2017}.

\begin{theorem}[variational backward error analysis]\label{thm:VarBEA}
	Consider a regular Lagrangian $L\in \mathcal{C}^a(TQ,\R)$ and a consistent variational discretisation $L_\Delta = I(L)$. There exists a formal power series $L_\mod$ in $h$, called {\em modified Lagrangian}, such that for any truncation $L^{[[k]]}_\mod$ after order $\mathcal{O}(h^k)$ the Euler-Lagrange equations $\mathcal{E}(L^{[[k]]}_\mod) = 0$ solved for $\ddot{q}$ coincide with the modified equation to $I(L)$ up to an error of order $\mathcal{O}(h^{k+1})$.
\end{theorem}

Let us review the construction of $L_\mod$ from $L_\Delta = I(L)$ for a given maximal truncation index $k$.

\begin{proof}[Construction]
	Let $\mathcal{L}^{[[k]]}_\Delta(q^{[m]})$ be a series expansion around $h=0$ of $L_\Delta(q(t-\frac{h}{2}),q(t+\frac{h}{2}),h)$ after order $\mathcal{O}(h^k)$. Define the meshed Lagrangian
	\begin{equation}\label{eq:Lmesh}
		\mathcal{L}^{[[k]]}_{\mathrm{mesh}} 
		= \sum_{j=0}^{\left\lfloor \frac{k}{2}\right\rfloor} \left(2^{(1-2j)} - 1\right) h^{2j} \frac{B_{2j}}{(2j)!}
		\frac{\d^{2j}}{\d t^{2j}}\mathcal{L}^{[[k]]}_\Delta(q^{[m]}),
	\end{equation}
	where terms of order $\mathcal{O}(h^{k+1})$ are truncated at the right hand side.
	Instances of $q^{(l)}$ with $l \ge 2$ are eliminated by iteratively substituting the modified equation for $L_\Delta$ and its derivatives and truncating terms of order $\mathcal{O}(h^{k+1})$. The resulting expression is the modified Lagrangian $L^{[[k]]}_\mod$.
\end{proof}

\begin{remark}
	The construction of $L^{[[k]]}_\mod$ is consistent with truncations, i.e.\ $L^{[[k]]}_\mod - L^{[[k+1]]}_\mod \in \mathcal{O}(h^{k+1})$.
\end{remark}

We now introduce inverse modified Lagrangians $L_\imod$ and prove their existence as formal power series in the discretisation parameter.

\begin{theorem}[inverse variational backward error analysis]\label{thm:Linvmod}
	Consider a regular Lagrangian $L\in \mathcal{C}^a(TQ,\R)$ and a consistent variational integrator $I$. There exists a formal power series $L_\imod$ in the step size $h$ such that for any truncation index $k$ the following statements hold.
	\begin{itemize}
		
		
		\item
		The modified equation to $I(L^{[[k]]}_\imod)$ is equivalent to the exact Euler--Lagrange equations $\mathcal{E}(L)=0$ up to terms of order $\mathcal{O}(h^{k+1})$.
		
		\item
		The modified Lagrangian $(L^{[[k]]}_\imod)_\mod$ to the variational discretisation of $I(L^{[[k]]}_\imod)$ coincides with $L$ up to terms of order $\mathcal{O}(h^{k+1})$.

	\end{itemize}

\end{theorem}

\begin{proof}
	The first two statements of the theorem follow directly from the last statement, which we now prove.
	Let $k$ be a truncation index. We consider the ansatz
	\begin{equation}\label{eq:Limodans}
		L^{[[k]]}_\imod = L(q,\dot{q}) + s L_1(q,\dot{q}) + \ldots + s^k L_k(q,\dot{q})
	\end{equation}
	for an additional formal variable $s$, where $L_1,\ldots,L_k$ need to be determined.
	When applied to the ansatz $L^{[[k]]}_\imod$, the construction of \cref{thm:VarBEA} yields $(L^{[[k]]}_\imod)^{[[k]]}_\mod$ which is a formal multivariate polynomial in $h$ and $s$.
	We then set $s=h$, truncate terms of order $\mathcal{O}(h^{k+1})$. We now show that one can derive analytic expressions for $L_1,\ldots,L_k$ by equating the coefficients of $h,h^2,\ldots,h^k$ successively with 0.
	
	The meshed Lagrangian $(L^{[[k]]}_\imod)^{[[k]]}_{\mathrm{mesh}}$ is of the form
	\begin{align*}
		\mathfrak{L}
		:=
		(L^{[[k]]}_\imod)^{[[k]]}_{\mathrm{mesh}}
		&= I(L^{[[k]]}_\imod) + h^2 \cdot  \mathcal{O}(s)
		= L(q,\dot{q}) + s L_1(q,\dot{q}) + \ldots + s^k L_k(q,\dot{q}) + (\mathcal{O}(h) +  h\cdot \mathcal{O}(s))
	\end{align*}
	since $I$ is consistent.
	For $l \le k$ let $\mathfrak{L}^{[[l]]}$ denote a truncation of the multivariate polynomial $\mathfrak{L}$ after terms of total degree $l$. 
	Observe that the multivariate polynomial $\mathfrak{L}^{[[l]]}$ contains only one instance of $L_l$. More precisely, $L_l(q,\dot{q})$ is the coefficient of $s^l h^0$.
	
	Using the regularity of $L$ and consistency of $I$, the modified equation for $I(L^{[[k]]}_\imod)$ is of the form
	\begin{equation}\label{eq:modEQLimod}
		\ddot{q} = \left( \frac{\p^2 L}{\p \dot{q} \p \dot{q}} \right)^{-1} \left( \frac{\p L}{\p q} - \frac{\p^2 L}{\p q \p \dot{q}} \dot{q} \right) + (\mathcal{O}(h) +   \mathcal{O}(s)).
	\end{equation}
	In this expression, $L_l$ appears only in summands which have total degree at least $l$. All appearances of $\ddot{q},q^{(3)},\ldots$ in $\mathfrak{L}$ are in summands of total degree at least 1. When
	$(L^{[[k]]}_\imod)^{[[k]]}_\mod$ is obtained from $\mathfrak{L}$ by repeated substitution of the right hand side of \eqref{eq:modEQLimod} and its derivatives, then a truncation of $(L^{[[k]]}_\imod)^{[[k]]}_\mod$ to total degree $l$ contains exactly one instance of $L_l(q,\dot{q})$ (the coefficient of $s^l h^0$).
	Now we set $s=h$ and truncate $(L^{[[k]]}_\imod)^{[[k]]}_\mod$ after terms of order $\mathcal{O}(h^k)$. 
	The discussion shows that equating the $l$th coefficient of $(L^{[[k]]}_\imod)^{[[k]]}_\mod$ with zero, we obtain an expression for $L_l(q,\dot{q})$ in terms of $L$, $L_1$, $\ldots$, $L_{l-1}$ and their derivatives. Setting $L_0 = L$, by an inductive argument, these expressions are analytic since $L$ is analytic. 
\end{proof}

The exact discrete Lagrangian $L^{\mathrm{exact}}_\Delta$ to a regular Lagrangian $L$ is defined as
\[
L^{\mathrm{exact}}_\Delta(q_0,q_1) = \int_0^h L(q(t),\dot{q}(t)) \d t,
\]
where $q$ solves the Euler--Lagrange equations $\mathcal{E}(L)=0$ with boundary conditions $q(0)=q_0$ and $q(h)=q_1$, where $q_0$, $q_1$ are sufficiently close such that the solution $q$ exists and is unique.

A reformulation of \cref{thm:Linvmod} yields

\begin{corollary}
	To a consistent variational integrator $I$ and a Lagrangian $L$ there exists an inverse modified Lagrangian $L_\imod$ such that the asymptotic expansion around 0 in the discretisation parameter $h$ of $I(L_\imod)$ and the exact discrete Lagrangian to $L$ coincide.
\end{corollary}

\begin{example}
	As a variational integrator, consider the midpoint rule $I$ defined by
	\[
	I(L)(q_0,q_1) = h L\left(\frac{q_0+q_1}{2},\frac{q_1-q_0}{h} \right).
	\]
	Using variational backward error analysis (\cref{thm:VarBEA}) and inverse variational backward error analysis (\cref{thm:Linvmod}) we obtain in the one-dimensional case
	\begin{align*}
		L_\mod &=L
		+\frac{1}{24} h^2 \left(\frac{\left(L^{(1,0)}-L^{(1,1)} \dot{q}\right)^2}{L^{(0,2)}}-L^{(2,0)} \dot{q}^2\right)
		+O\left(h^4\right)\\
		L_\imod &= L
		-\frac{1}{24} h^2 \left(\frac{\left(L^{(1,0)}-L^{(1,1)} \dot{q}\right)^2}{L^{(0,2)}}-L^{(2,0)} \dot{q}^2\right)
		+O\left(h^4\right).
	\end{align*}
	Here $L^{(i,j)} = \frac{\p^{i+j}}{ (\p q)^i  (\p \dot{q})^j } L(q,\dot{q})$.
	Higher order terms are more complicated and can be found along computational details in the corresponding GitHub repository.
\end{example}

\begin{example}
	Now consider the trapezoidal variational integrator $I$ defined by
	\[
	I(L)(q_0,q_1) = \frac{h}{2} L\left(q_0,\frac{q_1-q_0}{h} \right) + \frac{h}{2} L\left(q_1,\frac{q_1-q_0}{h} \right).
	\]
	The modified and inverse modified Lagrangians are given by
	\begin{align*}
		L_\mod &= L+ \frac{1}{24} h^2 \left(2 L^{(2,0)} \left(\dot{q}\right)^2+\frac{\left(L^{(1,0)}-L^{(1,1)} \dot{q}\right)^2}{L^{(0,2)}}\right)
		+O\left(h^4\right)\\
		L_\imod &= L-\frac{1}{24} h^2 \left(2 L^{(2,0)} \left(\dot{q}\right)^2+\frac{\left(L^{(1,0)}-L^{(1,1)} \dot{q}\right)^2}{L^{(0,2)}}\right)
		+O\left(h^4\right)\\
	\end{align*}
	Higher order terms are more complicated and can be found along computational details in the corresponding GitHub repository.
\end{example}

\section{Conclusion}

We have introduced the framework of Lagrangian Shadow Integration (LSI) and proved the existence of inverse modified Lagrangians in the sense of formal power series. 
The key novelty of LSI is that two types of discretisation errors are eliminated by design: (1) inverse modified Lagrangians can be learned directly from position data of trajectories. An approximation of velocities or acceleration data is not required. (2) Inverse modified Lagrangians compensate for the discretisation error of the integrator which is used to compute predictions of motions. Therefore, large step-sizes can be used while maintaining high accuracy.
Further, LSI does not require prior knowledge about the form of the Lagrangian such that the method applies very generally.
Moreover, the performance of LSI can be conveniently analysed using variational backward error analysis techniques. Additionally, system identification can be performed.

Future research directions are to embed symmetries into the learning process, for instance using GIM kernels or symmetric neural networks. An LSI prediction will then, by Noether's theorem, profit from exact conservation of conserved quantities related to variational symmetries, which should be beneficial especially for completely integrable systems.
It appears plausible that an introduction of structure into learning approaches for dynamical systems increases its resilience to noise in the learning data, reduces data requirements, and that the compensation of discretisation errors, as performed by LSI, remains important for noisy data since discretisation errors tend to be biased in contrast to noise.
A systematic investigation of these questions is part of future work.
Moreover, the statistical framework of Gaussian Processes could be used to obtain uncertainty estimations of predictions obtained using LSI with GPs. Furthermore, Lagrangian neural networks \cite{LNN} have successfully been applied to PDEs with variational structure. Practical applicability of LSI in this context can be investigated.
A more theoretical research direction is to transfer optimal truncation results for modified Lagrangians and Hamiltonians to inverse modified Lagrangian.



\bibliographystyle{elsarticle-num} 
\bibliography{bib}

\end{document}